\documentclass[letterpaper, 11pt]{article} 

\usepackage{graphics,graphicx}
\usepackage{parskip}
\usepackage{amsmath}
\usepackage{bbm}
\usepackage{amsthm}
\usepackage{mathtools}
\usepackage[english]{babel}
\usepackage[utf8]{inputenc}
\usepackage[title]{appendix}
\usepackage{amssymb} 
\usepackage{subcaption}
\usepackage[font=footnotesize,labelfont=small]{caption}
\usepackage{bbm}
\RequirePackage{geometry}
\usepackage[symbol]{footmisc}

\newtheorem{theorem}{Theorem}[section]
\newtheorem{lemma}[theorem]{Lemma}
\newtheorem{proposition}[theorem]{Proposition}
\newtheorem{corollary}[theorem]{Corollary}

\newtheorem{remark}{Remark}

\makeatletter
\DeclareRobustCommand\widecheck[1]{{\mathpalette\@widecheck{#1}}}
\def\@widecheck#1#2{%
	\setbox\z@\hbox{\m@th$#1#2$}%
	\setbox\tw@\hbox{\m@th$#1%
		\widehat{%
			\vrule\@width\z@\@height\ht\z@
			\vrule\@height\z@\@width\wd\z@}$}%
	\dp\tw@-\ht\z@
	\@tempdima\ht\z@ \advance\@tempdima2\ht\tw@ \divide\@tempdima\thr@@
	\setbox\tw@\hbox{%
		\raise\@tempdima\hbox{\scalebox{1}[-1]{\lower\@tempdima\box
				\tw@}}}%
	{\ooalign{\box\tw@ \cr \box\z@}}}
\makeatother

\newcommand{\ep}{\epsilon}
\newcommand{\codt}{\cdot}
\newcommand{\1}{\mathbb{I}}

\def\XXint#1#2#3{{\setbox0=\hbox{$#1{#2#3}{\int}$ }
\vcenter{\hbox{$#2#3$ }}\kern-.6\wd0}}

\DeclarePairedDelimiter\floor{\lfloor}{\rfloor}

\DeclareMathOperator{\sech}{sech}
\DeclareMathOperator{\Res}{Res}

\DeclareMathOperator{\sinc}{sinc}

\newcommand{\be}{\begin{equation}}
	\newcommand{\ee}{\end{equation}}
\newcommand{\bes}{\begin{equation*}}
	\newcommand{\ees}{\end{equation*}}
\newcommand{\mand}{\quad \text{and}\quad}
\newcommand{\R}{{\bf{R}}}
\newcommand{\E}{{\mathbb{E}}}

\newcommand{\Z}{{\bf{Z}}}
\newcommand{\N}{{\bf{N}}}
\newcommand{\e}{{\bf{e}}}

\newcommand{\bj}{{\bf{j}}}
\newcommand{\X}{{\bf{X}}}

\newcommand{\bk}{{\bf{k}}}

\newcommand{\y}{{\bf{y}}}
\newcommand{\x}{{\bf{x}}}
\newcommand{\z}{{\bf{z}}}
\newcommand{\0}{{\bf{0}}}
\renewcommand{\tilde}{\widetilde}
\renewcommand{\hat}{\widehat}
\renewcommand{\check}{\widecheck}

\newcommand{\norm}[1]{ \left\|#1\right\|}
\numberwithin{equation}{section}

\title{Macroscopic Wave Propagation for 2D Lattice with Random Masses}
\author{Joshua A. McGinnis\footnote{Drexel University Department of Mathematics, The Korman Center, 15 S 33rd St, Philadelphia, PA 19104}\footnote{ jam887@drexel.edu}}
\date{\today}
\begin{document}
	\maketitle
	\begin{abstract}
    We consider a simple 2D harmonic lattice with random, independent and identically distributed masses. Using the methods of stochastic homogenization, we prove that solutions with long wave initial data converge in an appropriate sense to solutions of an effective wave equation. The convergence is strong and almost sure. In addition, the role of the lattice's dimension in the rate of convergence is discussed. The technique combines energy estimates with powerful classical results about sub-Gaussian random variables.
\end{abstract}

	\section{Introduction}
	\label{Introduction}
	We prove an almost sure convergence result for solutions of the following two-dimensional spatially discrete, harmonic lattice with random masses in the long wave limit i.e. as the wavelength becomes longer:
	
	\be \label{Main2} m(\bj)\ddot{u}(\bj,t)=\Delta u(\bj,t)\ee 
	where $\bj=(j_1,j_2) \in \Z^2$. The discrete Laplacian is defined as 
	\be \Delta u(\bj,t)\coloneqq -4u(\bj,t)+\sum_{i=1}^{2}u(\bj+\e_i,t)+u(\bj-\e_i,t),\ee
	
where $\e_1=(1,0)$ and $\e_2=(0,1).$
	
The $m(j_1,j_2)$ are independent and identically distributed random variables (i.i.d.)\@ contained almost surely in some interval $[a,b] \in \R^+$. There is a long history of studying such lattices, especially when the masses or constants of elasticity vary periodically \cite{Brillouin}, since they can provide simplified models for investigating physical phenomena arising in more complex systems. We are interested in understanding to what degree approximations to the wave equation are still valid, since random coefficients could in principle be used to model impurities. The system is well understood when the masses are either constant or periodic with respect to $(j_1,j_2)$ \cite{Mielke}, but for a 2D lattice most of what is known for the random problem is for weak randomness  \cite{Spohn} or  numerical \cite{Porter}. Much has already been said about the continuous versions of \eqref{Main2}. For an extensive resource, see \cite{Papanicolau}. Note that rates of convergence are rarely addressed. Furthermore almost sure convergence results in the discrete setting require different techniques. In the continuous setting convergence can be achieved more directly through the law of large numbers. In our setting, one must define what is meant by convergence, and this is why we prove convergence through ``coarse-graining", which is used in \cite{Mielke} to prove convergence in the periodic problem. To achieve a rigorous rate of convergence, one must have bounds on the stochastic error terms. In \cite{McGinnis}, the law of the iterated logarithm was used. Here we use the theory of sub-Gaussian random variables.

For initial conditions whose wavelength and amplitude is $O(\ep^{-1})$ with $\ep$ a small positive number, we prove that the $\ell^2$ norm of the differences between true solutions and appropriately scaled solutions to the wave equation is almost surely $O(\ep^{-1-\sigma})$, where $\sigma$ is any small positive number. While such an absolute error diverges as $\ep \to 0^+,$ it happens that this is enough to establish an almost sure convergence of the macroscopic dynamics within the coarse-graining setting. 
	
	The article \cite{McGinnis} studies a similar problem on a 1D lattice. There the constants of elasticity also vary randomly and the system is studied in the relative coordinates. The so called multiscale method of homogeniziation, a by-now classical tool with a long history in PDE for deriving effective equations, see \cite{Cioranescu}, is employed. Our results should be contrasted with those in 1D. In \cite{McGinnis}, it is shown that that the coarse-graining converges at a rate of $\sqrt{\ep \log \log (\ep^{-1})}$, and this is thought to be sharp based on numerical evidence. Below, we show that in 2D, the rate of convergence is no slower than $\ep^{1-\sigma}$ when the masses are i.i.d. However the rate of convergence is different for layered masses i.e.  where masses are only random with respect to $j_1$ and constant with respect to $j_2$.  We explore this in Section \ref{Other Masses} and in the numerical experiments in Section \ref{Numerical Results}. In such a case we show that the convergence rate is comparable to the rate in the 1D lattice. We do not believe our estimates are sharp like those in the 1D setting because the analysis requires the use of a cut-off function, which demands slightly greater regularity and algebraic decay of the initial conditions. Still, numerical evidence suggests they are probably close to being sharp and likely only off by a logarithm and a vanishing power of $\ep$ introduced with the cut-off.
	
	Although large parts of the formal derivation are the same in 2D as they are in 1D, difficulty arises in the fact that, as far as we can tell, in 2D no known explicit formulation of a solution for \eqref{Chi Easy} exists. It can however be solved on a finite domain. Therefore a cut-off function involving $\ep$ is introduced and although it eventually disappears from the final Theorems \ref{Main Estimate} and \ref{Main Coarse Graining}, several new arguments need to be made, most important of which are almost sure estimates for solutions of \eqref{What chi solves} involving tails of sequences of sub-Gaussian random variables. In Section \ref{Deriving the Effective Equations}, we derive the effective equations. The core probability theory is in Section \ref{Probalistic Estimates}. A common energy estimate, akin to those found in \cite{Chirilus-Bruckner}, \cite{Gaison}, and \cite{Schneider} is given in  Section \ref{Lattice Energy Argument}. These help bound the error in terms of the ``residuals" defined in \eqref{Res}. Elementary theory of the energy of the effective wave, such as that in \cite{Evans}, is given in Section \ref{the effecive wave}. This theory is used to derive estimates of various norms needed in the following section, Section \ref{Bounding the Residual}. Finally, Section \ref{Main Estimate Section} and Section \ref{Coarse Graining} contains the main estimate and convergence result. The appendix contains technical proofs of several statements that are believable enough to be skipped in a first read through.

\section{Deriving the Effective Equations}
\label{Deriving the Effective Equations}
\subsection{Expansions}\label{Expansions}
Define the ``residual" to be 
\be\label{Res} \Res \tilde{u}(\bj, t)=m(\bj)\ddot{\tilde{u}}-\Delta\tilde{u}(\bj,t).\ee

We hope to find an approximate solution to \eqref{Main2} whose residual is as small as possible. We look for approximate solutions of the form

\be \label{ansatz} \tilde{u}(\bj,t)=\ep^{-1}U_0(\bj ,\ep \bj ,\ep t)+\ep U_2(\bj,\ep \bj, \ep t). \ee
The relative displacement of the masses is given by 
\be 
r_i(\bj,t)\coloneqq u(\bj+\e_i,t)-u(\bj,t).
\ee
Further down, we find that $U_0$ depends only on $\ep j$ and not on $j$ alone. Thus we have chosen as a convention the amplitude of the ansatz to be $O(\ep^{-1})$ so that the relative displacement of the masses is $O(1)$. This is the same convention chosen in \cite{Mielke}. Here 
\be U_i:\Z^2\times \R^2 \times \R \to \R, \ee
and we keep track of its arguments by letting $\X=\ep\bj$ and $\tau=\ep t,$ in which case $\X=(X_1,X_2) \in \R^2.$  In order to plug \eqref{ansatz} into \eqref{Main2}, we need to understand how $\Delta$ acts on functions of this form. Note that
\be \label{HowDacts} \Delta U (\bj, \X)=\sum_{i=1}^{2} U(\bj +\e_i,\ep \bj +\ep \e_i) +U(\bj -\e_i, \ep \bj -\ep \e_i)-2U(\bj, \ep \bj).  \ee
The addends,
\be U(\bj+\e_i, \ep \bj+\ep \e_i)+U(\bj-\e_i,\ep \bj -\ep \e_i),  \ee can be Taylor expanded in powers of $\ep$. Using this expansion in \eqref{HowDacts},
we have
\be \label{expansion}
\Delta U= \check{\Delta}_0U+\ep\check{\Delta}_1U+\ep^2\check{\Delta}_2U +O(\ep^3),
\ee
where
\be \label{O1} \check{\Delta}_0U(\bj,\X) \coloneqq -2U(\bj, \X)+\sum_{i=1}^2 U(\bj +\e_i, \X)+ U(\bj-\e_i, \X ) ,\ee  
\be \label{Oep}\check{\Delta}_1U(\bj,\X)\coloneqq\sum_{i=1}^2\partial_{X_i} U(\bj +\e_i, \X)- \partial_{X_i}U(\bj -\e_i,\X),\ee and 
\be \label{Oep2} \check{\Delta}_2U(\bj,\X)\coloneqq \frac{1}{2}\sum_{i=1}^2 \partial_{X_iX_i}U(\bj+\e_i,\X)+\partial_{X_iX_i}U(\bj -\e_i, \X). \ee 
Everything remaining in the expansion is $O(\ep^3)$ so we neglect it for now. Now we calculate $\Res{\tilde{u}}$ using \eqref{expansion}. Again grouping terms by powers of $\epsilon$ yields
\be\begin{aligned} \label{FirstRes} \Res \tilde{u}(\bj,\X,\tau)=&-\ep^{-1}\check{\Delta}_0U_0(\bj, \X,\tau)\\ & -  \check{\Delta}_1U_0(\bj -\e_i,\X, \tau)
\\& +\ep m(\bj)\partial_{\tau\tau}U_0(\bj,\X,\tau)-\ep \sum_{i=0}^1\check{\Delta}_{2i}U_{2-2i}(\bj,\X,\tau) \\& +O(\ep^2).
\end{aligned} \ee 
 Since we want $\Res \tilde{u}$ to be small, we solve for $U_0$ and $U_2$ so that formally the residual is $O(\ep^2).$ 
Assuming we seek bounded solutions, the $O(\ep^{-1})$ term implies
\be  \label{Delta is 0} \check{\Delta}_0U_0(\bj, \X,\tau) =0 .\ee We can meet \eqref{Delta is 0} by putting
\be \label{U0constant} U_0(\bj, \X,\tau) =\bar{U}_0(\X,\tau)\ee i.e. $U_0$ does not depend on the microscopic variable.
From \eqref{U0constant}, it is seen that 
\be \check{\Delta}_1U_0 \equiv 0 \ee and that 
\be\label{Macroscopic Laplacian}
\check{\Delta}_2U_0 = \Delta_\X U_0 \coloneqq \sum_{i=1}^2\partial_{X_iX_i} U_0.\ee 
With $U_0$ constant in $\bj$, the residual in \eqref{FirstRes} simplifies to 
\be\begin{aligned} \label{SecondRes} \Res \tilde{u}(\bj,\X,\tau)= & \ep m(\bj)\partial_{\tau\tau}U_0(\X,\tau)-\ep \Delta_{\X} U_0(\X,\tau)-\ep \check{\Delta}_{0}U_{2}(\bj,\X,\tau)+O(\ep^2).
\end{aligned}\ee
To make the $O(\ep)$ terms vanish, we would like
\be \label{To be Averaged} m(\bj)\partial_{\tau\tau}U_0(\X,\tau)-\Delta_\X U_0(\X,\tau)=\check{\Delta}_{0}U_{2}(\bj,\X,\tau). \ee 
Let $z(\bj)\coloneqq m(\bj)-\bar{m}$, and $\bar{m}\coloneqq\E[m(\bj)]$ where $\E$ is the expected value with respect to the probability measure of the i.i.d. lattice of masses. One way to solve \eqref{To be Averaged}
is to write it as 
\be \label{left and right} 
\bar{m}\partial_{\tau\tau}U_0(\X,\tau)-\Delta_\X U_0(\X,\tau)=\check{\Delta}_{0}U_{2}(\bj,\X,\tau)-z(\bj)\partial_{\tau\tau}U_0(\X,\tau),
\ee
and then pick $U_0$ and $U_2$ to force the left hand side and right hand side to vanish independently. From the left hand side, we find $U_0$ solves an effective wave equation. In this case, the wave speed $\bar{m}$ is guessed to be the correct one. For a more probabilistic derivation, see Subsection \ref{Averaging} below. 

Solving the right hand side by picking $U_2$ would be easy if only we had $\chi:\Z^2 \to \R$ s.t.
\be \label{Chi Easy}
\Delta \chi(\bj)=z(\bj).
\ee 
With such a solution we have \be\label{U2} 
U_2(\bj,\X,\tau)=\chi(\bj)\partial_{\tau\tau}U_0(\bj,\tau).
\ee
From this, the approximate solution would be 
\be\tilde{u}(\bj,\X,\tau)=\ep^{-1}U_0(\bj,\tau)+\ep \chi(\bj)\partial_{\tau\tau}U_0(\X,\tau). \ee
In order to control the magnitude of the residual, and to prove that $\ep^{-1}U_0$ is the dominant term in the approximation, we need to have asymptotic bounds on $\chi$. However, we do not even know how to solve for such a $\chi$. To circumnavigate this issue, we resort to solving $\eqref{Chi Easy}$ where the support of $z$ is on a finite domain dependent upon $\ep$. We call this restricted solution $\chi_r$. Justifying that  \be \label{Approximate Solution} \tilde{u}(\bj,\X,\tau)=\ep^{-1}U_0(\bj,\tau)+\ep \chi_r(\bj)\partial_{\tau\tau}U_0(\X,\tau),\ee
where $U_0$ solves the wave equation, is a valid approximation by finding almost sure asymptotic bounds on $\chi_r$ in $\ep$ is the main novelty of this paper.

\subsection{Averaging}\label{Averaging}
Another way to get at \eqref{To be Averaged} is by averaging. We make another ad hoc assumption, this time on $U_2$. We require that
\be\label{assumption 2} \lim_{|\bj|_\infty \to \infty} \dfrac{U_2(\bj,\X,\tau)}{\lvert \bj\rvert_\infty} =0 \ee where 
\be \label{infinity norm}
\lvert\bj\rvert_\infty \coloneqq \max\{|j_1|,|j_2|\ | \ \bj=(j_1,j_2)\}.
\ee 
We have introduced the $\infty$ norm here for technical reasons. Henceforth $\0$ is the vector $(0,0).$ We define disks 
\be\label{Disk} D(\bj_0,r)\coloneqq \{\bj \in \Z^2 \ | \ |\bj-\bj_0|_\infty \leq r\}. \ee The boundary of these disks are given by 
\be\label{Boundary}\delta D(\bj_0,r)=\{\bj \ | \ |\bj -\bj_0|_\infty =r \}.\ee
For any set $U \subset \Z^2$, $\lvert U \rvert$ means the number of elements in the set. With the assumption \eqref{assumption 2}, one may argue that for $U_2$ satisfying \eqref{To be Averaged} to exist for all $\X$ and $\tau$, it must be that its spatial average is $0$ i.e.
\be \label{Averaged RHS}\lim_{r \to \infty}\frac{1}{|D(\0,r)|}\sum_{\bj \in D(\0,r)}m(\bj)\partial_{\tau\tau}U_0(\X,\tau)-\Delta_{\X} U_0(\X,\tau)=0,\ee but by the law of large numbers
\be\label{Averaged LHS} \begin{aligned} &\lim_{r \to \infty}\frac{1}{|D(\0,r)|}\sum_{\bj \in D(\0,r)}m(\bj)\partial_{\tau\tau}U_0(\X,\tau)-\Delta_{\X} U_0(\X,\tau)\\=&\bar{m}\partial_{\tau\tau}U_0(\X,\tau)-\Delta_\X U_0(\X,\tau).\end{aligned} \ee

Combining \eqref{Averaged RHS} with \eqref{Averaged LHS} yields the same effective wave equation as seen in the left hand side of \ref{left and right}, 
\be\label{The efffective equation} \bar{m}\partial_{\tau\tau}U_0(\X,\tau)-\Delta_\X U_0(\X,\tau)=0.\ee 
 We now subtract \eqref{The efffective equation} from \eqref{To be Averaged} yielding
\be\label{Fluctuations}z(\bj)\partial_{\tau\tau}U_0(\X,\tau)=\check{\Delta}_0U_2(\bj,\X,\tau) \ee which is the right hand side of \eqref{left and right}.
\section{Probabilistic Estimates}
\label{Probalistic Estimates}
The intuition for the next step is that we need only solve \eqref{Chi Easy} on a domain which is accessible to the wave i.e. since we only care about $|t| \leq T\ep^{-1}$, the approximation works without having a solution to \eqref{Chi Easy} outside a radius proportional to $T\ep^{-1}$. Given enough decay of the initial conditions, only a negligible amount of the wave will have traveled outside this radius. Therefore we can use a spatial cutoff and well known results of sub-Gaussian random variables to obtain bounds on $\chi$. The form of the estimates in this section are motivated by the error estimates in the following sections, specifically Section \ref{Bounding the Residual}. 
\subsection{Poisson Problem with Random Source}\label{Poisson Problem with Random Source}
Let $I:\Z^2\to \{0,1\}$ s.t.
\be \label{K delta}
I(\bj) =\begin{cases} 1 & \bj =\textbf{0} \\ 0 & \bj \neq \textbf{0}
\end{cases}
\ee
and 
$\varphi:\Z^2 \to \R$ s.t. 
\be \label{Fundamental Solution}
\Delta \varphi(\bj)= I(\bj).
\ee
Then $\varphi$ is the fundamental solution and it is known that $\varphi(\0)=0$, and for $\bj \neq 0$
\be 
\label{Fundamental Solution Explicit}
\varphi(\bj)=\frac{1}{2\pi}\log|\bj|+C_0+O(|\bj|^{-2}), \ |\bj| \to \infty
\ee
where $|\bj|\coloneqq \sqrt{j_1^2+j_2^2}.$ The proof of the existence and uniqueness of $\varphi$ can be found in \cite{Duffin}. In \eqref{Fundamental Solution Explicit}, $\log$ may be replaced with $\log^+$. It is defined by $\log^+(x)=\max\{0,\log(x)\}$ and $\log^+(0)=0$, so all logs in the sequel should be thought of as $\log^+$, even if we neglect the plus sign.

Define 
\be\label{Chi restricted}
\chi_r(\bj) \coloneqq \sum_{\bk \in D(\0,r)}\varphi(\bj-\bk)z(\bk) \ee
where the $z(\bk)$ and $D(\0,r)$ are defined in Subsection \ref{Expansions}. The following is true
\be 
\label{What chi solves}
\Delta\chi_r(\bj) = \begin{cases} z(\bj) & \bj \in D(\0,r) \\ 
0 & \bj \in \Z^2/D(\0,r)
\end{cases}.
\ee
For any generic linear operator $\mathcal{L}$ acting on functions of $\Z^2$, we have 
\be\label{Reconsider} \mathcal{L}(\chi_r)(\bj)=\sum_{\bk \in D(\0,r)} \mathcal{L}(\varphi)(\bj-\bk)z(\bk) .\ee
\subsection{Estimates on the Solution}\label{Estimates on the Solution}
We make use of Hoeffding's Inequality, the justification of which can be found in \cite{Massart}. We state the theorem here for completeness. 
\begin{theorem}(Hoeffding's Inequality, 1962)
Let $q_1,q_2,...,q_n$ be independent random variables such that $q_i \in [a_i,b_i]$ almost surely for all $i \leq n$. Let 
\be 
S \coloneqq \sum_{i=1}^{n}q_i-\E[q_i].
\ee
Then for any positive $t$, 
\be 
P\left(S \geq t \right) \leq \exp \left(\frac{-2t^2}{\sum_{i=1}^{n}(a_i-b_i)^2}\right).
\ee
\end{theorem}
 Recall that $z(\bk) \in [a-\bar{m},b-\bar{m}]$ almost surely and is mean $0$. Therefore, $z(\bk)\mathcal{L}(\varphi)(\bj-\bk) \in[\mathcal{L}(\varphi)(\bj-\bk)(a-\bar{m}),\mathcal{L}(\varphi)(\bj-\bk)(b-\bar{m})] $ and is mean $0$. (It might be that $\mathcal{L}(\varphi)(\bj-\bk)$ is negative in which case the bounds of the interval are flipped.)
Denote
\be \label{Restricted Norm}
\norm{\mathcal{L}\varphi}^2_{D(\bj,r)}\coloneqq \sum_{\bk \in D(\bj,r)}(\mathcal{L}(\varphi)(\bk))^2=\sum_{\bk \in D(0,r)}(\mathcal{L}(\varphi)(\bj-\bk))^2.
\ee
Application of Hoeffding's Inequality to $\mathcal{L}(\chi_r)$ and $-\mathcal{L}(\chi_r)$ as defined in \eqref{Reconsider} yields
\be \label{Tail Estimate}
P(|\mathcal{L}(\chi_r)(\bj)| \geq t) \leq 2\exp\left( \frac{-2t^2}{(a-b)^2\norm{\mathcal{L}\varphi}^2_{D(\bj,r)}}\right).
\ee
To make use of \eqref{Tail Estimate}, we order the $\chi_r(\bj)$ into a sequence. First, restrict $r$ to take on only values $\N^+/\{1\}$. We map $r$ and $\bj$ to a set of indices using a bijection $I: \Z^2\times \N^+/\{1\} \to \N^+$ which satisfies the inequality \be 
\label{Index Relation}
I(\bj, r) \leq C\max\{r^3,|\bj|^3\}.
\ee The possibility of such a bijection is proven below in Lemma \ref{Bijection Existence}. We let 
\be\mathcal{L}(\chi)(n)\coloneqq \mathcal{L}(\chi_r)(\bj) \ \text{and} \ \norm{\mathcal{L}\varphi}^2_{D(n)}\coloneqq \norm{\mathcal{L}\varphi}^2_{D(\bj,r)}\ee
 when $(\bj,r)=I^{-1}(n)$. One sees we still have \eqref{Tail Estimate} as 
 \be P(|\mathcal{L}(\chi)(n)| \geq t) \leq 2\exp\left( \frac{-2t^2}{\norm{\mathcal{L}\varphi}^2_{D(n)}(a-b)^2}\right). \ee
 Using the Borel-Cantelli Lemma, which one may find in \cite{Durrett}, we have that for all but finitely many $n$ that 
 \be \label{Borel-Cantelli}
 |\mathcal{L}(\chi)(n)| \leq \sqrt{(a-b)^2\norm{\mathcal{L}\varphi}^2_{D(n)}\log(n)}
 \ee 
 almost surely. \eqref{Borel-Cantelli} and \eqref{Index Relation} imply that there exists $C_\omega$ almost surely s.t. 
 \be \label{Almost Sure Chi}
 |\mathcal{L}(\chi_r)(\bj)| \leq C_\omega \sqrt{\norm{\mathcal{L}\varphi}^2_{D(\bj,r)}\left(\log(|\bj|)+\log(r)\right)}.
 \ee
 For $\varphi$ given by \eqref{Fundamental Solution Explicit} we find that for some constant $C$
\be\label{Norm of phi}
\norm{\varphi}^2_{D(\bj,r)} \leq Cr^2\log(|\bj|+r)^2.
\ee

A pair of $\mathcal{L}$ operators we need to consider is
\be \label{Shifts}
S_i^\pm(\xi)(\bj)\coloneqq\xi(\bj \pm \e_i).
\ee
Notice that for $S_i^\pm$, there exists a $C$ s.t.
\be\label{Norm of shift phi}
\norm{S^\pm_i\varphi}^2_{D(\bj,r)} \leq Cr^2\log(|\bj|+r)^2.
\ee
Another $\mathcal{L}$ operator we need to consider is $\delta_i$ which we define as 
\be\label{Partial Center Difference}
\delta_i (\varphi)(\bj)\coloneqq\varphi(\bj +\e_i)-\varphi(\bj-\e_i).
\ee
From Lemma \ref{Bound on delta phi}, there exists a $C$ s.t. 
\be 
|\delta_i(\varphi)(\bj)| \leq \frac{C}{|\bj|+1}.
\ee
From Lemma \ref{Bound on norm delta phi} there exists another constant $C$ s.t. 
\be
\label{Norm of delta phi}
\norm{\delta_i\varphi}^2_{D(\bj,r)} \leq C\log(|\bj|+r).
\ee
From \eqref{Almost Sure Chi}, \eqref{Norm of phi}, and \eqref{Norm of delta phi} we get the following theorem.
\begin{theorem}
Let $\{z(\bj)\}_{\bj \in \Z^2}$ be independent, mean zero random variables with $z(\bj) \in [a-\bar{m},b-\bar{m}]$ a.s. Let $r \geq 2$ be any real number, and $\chi_r(\bj)$ be as defined in \eqref{Chi restricted} with $\varphi$ defined by \eqref{Fundamental Solution}. Then there almost surely exists a constant $C_\omega$ s.t. 
\be 
\label{Main Prob Thm 0}
|\chi_r(\bj)| \leq C_\omega r\left(\log(|\bj|)^{\frac{3}{2}}+\log(|r|)^{\frac{3}{2}}\right),
\ee 
\be 
\label{Main Prob Thm 1}
|S_i^\pm(\chi_r)(\bj)| \leq C_\omega r\left(\log(|\bj|)^{\frac{3}{2}}+\log(|r|)^{\frac{3}{2}}\right)
\ee 
and 
\be 
\label{Main Prob Thm 2}
|\delta_i(\chi_r)(\bj)| \leq C_\omega \left(\log(|\bj|)+\log(|r|)\right).
\ee 

\end{theorem}
 \begin{remark}\label{Floor}
 Note that the definition in \eqref{Chi restricted} implies $\chi_{r}=\chi_{\floor{r}}$, and thus the theorem holds for all real $r\geq 2$, even though $r$ was originally an integer in the bijection in \eqref{Index Relation}.
 \end{remark}
\begin{remark}
The $\omega$ indicates that the constant depends on the realization of the masses. 
\end{remark}
\begin{proof}
Insert \eqref{Norm of shift phi} and \eqref{Norm of delta phi} into \eqref{Almost Sure Chi}. After sweeping any constants into $C_\omega$, one has 
\be 
|\chi_r(\bj)| \leq C_\omega r \log(|\bj|+r)\sqrt{\left(\log(|\bj|)+\log(r)\right)}
\ee
and 
\be 
|\delta_i(\chi_r)(\bj)| \leq C_\omega  \sqrt{\log(|\bj|+r)\left(\log(|\bj|)+\log(r)\right)}.
\ee
An elementary argument in Lemma \ref{Log Subadditive} yields
\be
|\chi_r(\bj)| \leq C_\omega r \left(\log(|\bj|)+\log(r)\right)^{\frac{3}{2}}
\ee
and 
\be
|\delta_i(\chi_r)(\bj)| \leq C_\omega  \left(\log(|\bj|)+\log(r)\right).
\ee
A typical argument using convexity of $(\cdot)^{\frac{3}{2}}$ and sweeping constants into $C_\omega$ now gives the result.
\end{proof}
\section{Lattice Energy Argument}
\label{Lattice Energy Argument}
\subsection{Error Estimate}
\label{Error Estimate}
 We want the approximate solution to be good for $t$ s.t. $|t| \leq \frac{T}{\ep}$.  Recall that the wave travels at an effective speed given by 
 \be\label{def of speed} c \coloneqq (\bar{m})^{-\frac{1}{2}} \ee according to \eqref{The efffective equation}. Therefore, let
\be \label{The Radius}
R_\ep \coloneqq \frac{cT+\ep^{-\sigma}}{\ep}+1
\ee
 where $\sigma$ is any small positive real number.  Recall that $\tau=\ep t$ and $\X=\ep \bj$. Then the approximate solution has the form of \eqref{Approximate Solution} and is
\be 
\label{Approximate Solution 2}
\tilde{u}(\bj,t)=\ep^{-1}U(\X,\tau)+\ep \chi_{R_\ep}(\bj)\partial_{\tau\tau}U(\X,\tau),
\ee
where $U$ satisfies \eqref{The efffective equation} and $\chi_{R_\ep}$ is defined by \eqref{Chi restricted} so that $\chi_{R_\ep}$ solves \eqref{What chi solves}. Define the ``error"
\be \label{error}
\xi(\bj,t)\coloneqq  \tilde{u}(\bj,t)-u(\bj,t).
\ee
Using \eqref{Main2} we find
\be \label{xi doubledot}
m(\bj)\ddot{\xi}(\bj,t)=\Delta \xi (\bj ,t)+\Res{\tilde{u}}(\bj,t)
\ee
where $\Res{\tilde{u}}$ is given by $\eqref{Res}.$ Before going further it is useful to define some other operators similar to $\delta_i$ given by \eqref{Partial Center Difference}. We also have the two partial difference operators 
\be 
\label{Partial differences}
\delta_i^\pm(\xi)(\bj)\coloneqq \pm (\xi(\bj\pm\e_i)-\xi(\bj)).
\ee

The energy of the error is given by 
\be \label{Lattice Energy}
H(t)\coloneqq \frac{1}{2}\sum_{\bj \in \Z^2}m(\bj)\dot{\xi}(\bj,t)^2+(\delta^+_1\xi(\bj,t))^2+(\delta^+_2\xi(\bj,t))^2.
\ee
Differentiating gives
\be 
\dot{H}(t)=\sum_{\bj \in \Z^2}\left(m(\bj)\dot{\xi}(\bj,t)\ddot{\xi}(\bj,t)+\sum_{i=1}^{2}\delta^+_i\xi(\bj,t)\delta^+_i\dot{\xi}(\bj,t)\right).
\ee
We insert \eqref{xi doubledot} 
\be 
\label{Unnameable 1}
\dot{H}(t)=\sum_{\bj \in \Z^2}\left(\dot{\xi}(\bj,t)(\Delta \xi (\bj ,t)+\Res{\tilde{u}}(\bj,t))+\sum_{i=1}^{2}\delta^+_i\xi(\bj,t)\delta^+_i\dot{\xi}(\bj,t)\right).
\ee
Via summation by parts
\be \label{Summation by Parts} 
\sum_{\bj \in \Z^2}\left(\dot{\xi}(\bj,t)\Delta \xi (\bj ,t)+\sum_{i=1}^{2}\delta^+_i\xi(\bj,t)\delta^+_i\dot{\xi}(\bj,t)\right)= 0.
\ee
Hence \eqref{Unnameable 1} with Cauchy-Schwarz becomes
\be \label{Unnameable 2}
\dot{H}(t)=\sum_{\bj \in \Z^2}\dot{\xi}(\bj,t)\Res{\tilde{u}}(\bj,t) \leq \norm{\dot{\xi}(\cdot,t)}_{\ell^2}\norm{\Res{\tilde{u}}(\cdot,t)}_{\ell^2}.
\ee
The norm $\norm{\dot\xi,\delta^+_1 \xi, \delta^+_2 \xi}_{\ell^2}^2$ and the energy $H$ are equivalent, and so there exists a constant $C$ depending on $a$ and $b$ s.t. \eqref{Unnameable 2} becomes
\be \label{Differential Inequality for H}
\dot{H}(t) (H(t))^{\frac{-1}{2}}\leq C\norm{\Res{\tilde{u}}(\cdot,t)}_{\ell^2}.
\ee
Integrating from $0$ to $t$ for any $|t| \leq \frac{T}{\ep}$ gives
\be \label{Inequality for H}
\left(H(t)\right)^{\frac{1}{2}} \leq C\left(H(0)\right)^{\frac{1}{2}}+\ep^{-1}CT\sup_{|t| \leq \ep^{-1}T}\norm{\Res{\tilde{u}}(\cdot,t)}_{\ell^2}.
\ee
By the equivalence of norms and \eqref{Inequality for H}
\be \label{Inequality of the Error}
\norm{\dot{\xi}(\cdot,t), \delta_1^+\xi(\cdot,t),\delta_2^+\xi(\cdot,t)}_{\ell^2} \leq C\norm{\dot{\xi}(\cdot,0), \delta_1^+\xi(\cdot,0),\delta_2^+\xi(\cdot,0)}_{\ell^2} +\ep^{-1}CT\sup_{|t| \leq \ep^{-1}T}\norm{\Res{\tilde{u}}(\cdot,t)}_{\ell^2}.
\ee
The bound \eqref{Inequality of the Error} is valid for all $t$ s.t. $|t| \leq \frac{T}{\ep}$. Let us suppose that
\be \label{Assumption on Disparity of Initial Conditions} 
\norm{\dot{\xi}(\cdot,0), \delta_1^+\xi(\cdot,0),\delta_2^+\xi(\cdot,0)}_{\ell^2} \leq \ep^{-1}CT\sup_{|t| \leq \ep^{-1}T}\norm{\Res{\tilde{u}}(\cdot,t)}_{\ell^2}.
\ee
Furthermore let 
\be \label{Velocity and Relative Displacement}
p(\bj,t)\coloneqq \dot{u}(\bj,t)\ \text{and} \ r_i(\bj,t)\coloneqq\delta_i^+u(\bj,t)
\ee
with the analogous definitions of $\tilde{p}$ and $\tilde{r}_i$. These are the velocities and relative displacements of the masses.
Then, for some constant $C$ dependent on $T$, \eqref{Inequality of the Error} becomes \be \label{Inequality for Microstate} 
 \sup_{|t| \leq \ep^{-1} T}\norm{p-\tilde{p},r_1-\tilde{r}_1,r_2-\tilde{r}_2}_{\ell^2} \leq \ep^{-1}C\left(\sup_{|t| \leq \ep^{-1}T} \norm{\Res{ \tilde{u}}}_{\ell^2}\right).
\ee
 From \eqref{Inequality for Microstate} and \eqref{Velocity and Relative Displacement}, we also have by integrating that
\be \label{Inequality for u}
\sup_{|t| \leq \ep^{-1}T}\norm{u-\tilde{u}}_{\ell_2} \leq \ep^{-2}C\left( \sup_{|t| \leq \ep^{-1}T} \norm{\Res{ \tilde{u}}}_{\ell^2}\right).
\ee
Controlling the residual is essential to proving our main theorem. Thus we are concerned with proving the following proposition throughout the next couple of sections.
\begin{proposition} \label{Res bound prop}
Suppose the $m(j)$ are i.i.d. random variables contained in some interval $[a,b] \subset \R^+$ almost surely. Then almost surely there exists a constant $C_\omega(T,a,b, \bar{m})$ s.t. \eqref{Res Bound}
\be 
\label{Res Bound}
\sup_{|t| \leq \ep^{-1} T} \norm{\Res \tilde{u}}\leq \ep^{1-\sigma}\log (\ep^{-1-\sigma})^{\frac{3}{2}}C_\omega\norm{\Delta \phi, \psi}_{H^5_\sigma}.
\ee
\end{proposition}
The norm $\norm{\Delta \phi, \psi}_{H^5_\sigma}$ is a norm on the initial conditions of $U(\X,\tau)$ given in Subsection \ref{Initial Conditions}. The norm is defined in Subsection \ref{Tail Energy Sec}. Crucially, it does not depend on $\ep$.

\subsection{Calculating the Residual}
Now we calculate \eqref{Res} for our approximate solution given by \eqref{Approximate Solution 2}. We do not write the arguments of functions when we do not need, but recall that $\X=\ep j$ and $\tau=\ep t.$ Recall the operators defined in  \eqref{Shifts}, \eqref{Partial Center Difference}, and \eqref{Partial differences}. We also have 
\be \label{Partial Laplacian}
\Delta_i u \coloneqq S_i^+u-2u+S_i^-u
\ee
and note that 
\be \label{Obviously}
\Delta u=\Delta_1u+\Delta_2 u.
\ee
Next calculate using Lemma \ref{Product Rule} that
\be\label{Expansion od delta u tilde}
\Delta \tilde{u}= \ep^{-1}\Delta U + \ep\Delta(\chi_{R_\ep})U_{\tau\tau}+ \ep\sum_{i=1}^2\delta_i(\chi_{R_\ep})\delta_i^-(U_{\tau\tau})+S_i^+(\chi_{R_\ep})\Delta_i(U_{\tau\tau}). \ee
For $D(\0,R_\ep) \subset \Z^2$, we denote the indicator of $D(\0,R_\ep)$ as $\1_{D(\0,R_\ep)}$. The complement of $D(\0,R_\ep)$ is $D(\0,R_\ep)^c.$ Using  \eqref{What chi solves} in \eqref{Expansion od delta u tilde}
\be\label{Unnameable 3}
\Delta \tilde{u}= \ep^{-1}\Delta U + \ep z\1_{D(\0,R_\ep)}U_{\tau\tau}+ \ep \sum_{i=1}^2\delta_i(\chi_{R_\ep})\delta_i^-(U_{\tau\tau})+S_i^+(\chi_{R_\ep})\Delta_i(U_{\tau\tau}). \ee
Recall that $\z(\bj)=m(\bj)-\bar{m}$ and that $U$ solves \eqref{The efffective equation}. Then \eqref{Unnameable 3} becomes
\be 
\Delta \tilde{u}=\ep^{-1}\Delta U-\ep \Delta_\X U +\ep\left(mU_{\tau\tau} -z\1_{D(\0,R_\ep)^c}U_{\tau\tau}+ \sum_{i=1}^2\delta_i(\chi_{R_\ep})\delta_i^-(U_{\tau\tau})+S_i^+(\chi_{R_\ep})\Delta_i(U_{\tau\tau})\right)
\ee 
where $D(\0,R_\ep)^c $ denotes the complement of the original set. Plugging into \eqref{Res} we have 
\be \begin{aligned} \label{Final Form of Res}
\Res{\tilde{u}}=&\ep^{-1}\left(\ep^2 \Delta_\X U-\Delta U\right)+\ep \left(z\1_{D(\0,R_\ep)^c}U_{\tau\tau}- \sum_{i=1}^2\delta_i(\chi_{R_\ep})\delta_i^-(U_{\tau\tau})+S_i^+(\chi_{R_\ep})\Delta_i(U_{\tau\tau})\right)\\&+\ep^3m\chi_{R_\ep}U_{\tau\tau\tau\tau}.
\end{aligned}
\ee

\label{Calculating the Residual}
\section{The Effective Wave}\label{the effecive wave}
\subsection{Initial Conditions}
We specify the initial conditions for \eqref{Main2}. For smooth enough functions $\phi, \psi:\R^2 \to \R$, we let
\be \label{Initial Conditions}
u(\bj,0)=\ep^{-1}\phi(\ep\bj) \ \text{and} \ \dot{u}(\bj,0)=\psi(\ep\bj). 
\ee
With these initial conditions, the initial relative displacements and velocity as defined in \eqref{Velocity and Relative Displacement} are roughly $\ep^{-1}$ in $\ell^2$. 
If the initial conditions for $\tilde{u}$ are defined analogously to \eqref{Initial Conditions} with $\tilde{\phi}$ and $\tilde{\psi},$ then in order to satisfy \eqref{Assumption on Disparity of Initial Conditions}, we need informally that 
\be 
\norm{\phi-\tilde{\phi}, \psi-\tilde{\psi}}_{\ell^2}\leq  \ep C. 
\ee 
Let us suppose that \eqref{Assumption on Disparity of Initial Conditions} is satisfied so that we do not need to worry about the disparity of the initial conditions. Therefore, to save on writing, $\phi$ and $\psi$ can simply refer to the initial conditions of either the true or approximate solution for now.

Suppose we have the following
\be \label{Macroscopic Initial Conditions}
U(\X,0)= \phi(\X) \ \text{and} \ \partial_{\tau} U(\X,0)=\psi(\X).
\ee
Then $\ep^{-1}U(\X,\tau)$ satisfies \eqref{The efffective equation}, and so according to  \eqref{Approximate Solution 2} but excluding the higher order term, which we show later is small enough to ignore, we have
\be  \label{Approximate Solution 3}
\tilde{u}(\bj, t) =\ep^{-1}U(\X ,\tau).
\ee 
This is the correct approximate solution in that it satisfies \eqref{The efffective equation} and \eqref{Initial Conditions}. 

\subsection{The Energy} \label{Energy}
Recall that $c$ is the effective wave speed defined in \eqref{def of speed}. In order to bound the terms in \eqref{Final Form of Res}, we need to control $L^2$ norms of derivatives of $U$ in terms of the initial conditions. This can be done using arguments involving the energy of $U$. For $V(\X,\tau):\R^{2} \times \R \to \R^{2^k}$, let $DV$ be the total derivative of $V$ with respect to its $2$ spatial components. In particular $D$ is the gradient, $\nabla$, of scalar valued functions. For each $j \in \{1,2,3, \dots, 2^k\},$ define \be \label{Continuum Energy} 
E(V_j)(\tau)\coloneqq\frac{1}{2}\int_{\R^2}\left(\partial_\tau V_j \right)^2+c^2\left|\nabla\left(V_j\right)\right|^2d\X.\ee Then $E(V)(\tau)$ represents a $2^k$ dimensional vector of energies. For $U$ satisfying \eqref{The efffective equation} with initial conditions \eqref{Macroscopic Initial Conditions}, it is well known, see for instance \cite{Craig} or \cite{Evans}, that for all $\tau$ that
\be \label{Basic Bound on Energy} 
\left|E\left(D^k\frac{\partial^{i}U}{\partial\tau^i}\right)(\tau) \right|\leq C\left(\norm{\nabla\phi}_{H^{i+k}}^2+\norm{\psi}_{H^{i+k}}^2\right) .\ee 
The proof is also provided in  Lemma \ref{Standard Energy Result} in the Appendix. The energy of derivatives of $U$ \eqref{Continuum Energy} is equivalent to $L^2$ norms of derivatives of $U$. 
\be \label{Basic Energy Bound 1}
\sum_{j=1}^k\norm{D^j\frac{\partial^iU}{\partial \tau^i}}^2_{L^2} \leq C\sum_{j=0}^{k-1} \left|E\left(D^j\frac{\partial^{i}U}{\partial\tau^i}\right)(\tau) \right|\leq C\left(\norm{\nabla\phi}^2_{H^{i+k-1}}+\norm{\psi}^2_{H^{i+k-1}}\right).
\ee 
The constant depends on $\bar{m}$, and we see that $\nabla\phi \in H^{i+k-1}$ and $\psi \in H^{i+k-1}$ is needed.

\subsection{Weighted Energy}\label{Weighted Energy section}
In \eqref{Final Form of Res}, there are terms where $U$ is ``weighted'' by some version of $\chi$. By \ref{Main Prob Thm 1}, the spatially varying aspect of these can be bounded by functions involving logarithms. We need a way to deal with these weights in a continuous setting. We can obviously find a smooth function $w(\X):\R^d \to \R$ s.t. 
\be \label{Assumption 1 on Weight}
w(\X)=\log(|\X|+1)^{\frac{3}{2}}+1 \ \forall \ |\X| \geq 1 \ \text{and} \ w(\X) \geq 1, 
\ee
and for all $n \in \N^+$ there exists a constant $W$ s.t.
\be \label{Assumptions 2 on Wieght} 
\left|D^{n}w\right| \leq W \mand \left|D(w^2) \right| \leq W.\ee In the same context as \eqref{Continuum Energy}, the possibly vector valued ``weighted energy" is given by 
\be \label{Weighted Energy} 
E_{w}(V_j)(\tau)\coloneqq \frac{1}{2}\int_{\R^2}w^2\left(\partial_\tau V_j\right)^2+c^2w^2\left|\nabla\left(V_j\right)\right|^{2}d\X.\ee According to Lemma \ref{Weighted Energy result}, for all $\tau$
\be \label{Bound on Weighted Energy almost}  
\left|E_w\left(D^k\frac{\partial^{i}U}{\partial \tau^{i}}\right)(\tau)\right| \leq  (|\tau|+1) C\left(\norm{\nabla \phi}^2_{H_w^{i+k}}+\norm{\psi}^2_{H_w^{i+k}}\right),
\ee 
where from \eqref{Assumption 1 on Weight}
\be \label{Def of Weighted Norm w}
\norm{\psi}_{H^k} \leq \norm{\psi}_{H^k_w}\coloneqq \sum_{j=0}^k\norm{wD^j\psi}_{L^2}.
\ee
The constant depends on $\bar{m}$ as well as $W$. Assumptions \eqref{Assumption 1 on Weight} and \eqref{Assumptions 2 on Wieght} on $w$ also give us
\be \label{Distributing Derivative} 
\norm{D^{k}\left(w\frac{\partial^iU}{\partial\tau^{i}}\right)}_{L^2} \leq C\sum_{j=0}^{k}\norm{wD^j\frac{\partial^iU}{\partial\tau^{i}}}_{L^2}\leq C\sum_{j=0}^{k}\sqrt{\left|E_w\left(D^j\frac{\partial^{i-1}U}{\partial\tau^{i-1}}\right)\right|}.\ee 
One sees from \eqref{Bound on Weighted Energy almost} and \eqref{Distributing Derivative} that for $i \geq 1$,
\be \label{Bound on Weighted Energy}
\norm{w\frac{\partial^iU}{\partial\tau^{i}}}_{H^k}\leq C\sqrt{|\tau|+1} \left(\norm{\nabla \phi}_{H_w^{i+k-1}}+\norm{\psi}_{H_w^{i+k-1}}\right)
\ee 
which holds for all $\tau.$ We note that we choose the convention that $i\geq 1$ since if both $i$ and $k$ are $0$, we have just $U$, which we cannot estimate in terms of energy.
\subsection{Tail Energy}\label{Tail Energy Sec}
We have already defined disks, so to be very clear, we let 
\be\label{Definition of Ball} 
B(r)= \{\X \in \R^2 \ | \ |\X| < r \}.
\ee

Recall that $U$ satisfies \eqref{The efffective equation} and travels with speed $c$. In the same context as \eqref{Energy}, the energy at the tails of the function is given by 
\be \label{Tail Energy} 
\tilde{E}(V)(\tau)\coloneqq \int_{B(c|\tau|+\ep^{-\sigma})^c}\left(\partial_\tau V_j\right)^2+c^2\left|\nabla \left( V_j\right)\right|^2d\X.
\ee
Lemma \ref{Tail Energy Lemma} shows that for $i \geq 1$
\be 
\left|\tilde{E}\left(D^k\frac{\partial^i U}{\partial\tau^i}\right)(\tau)\right| \leq C\left(\norm{\nabla \phi}^2_{H^{i+k}(B(\ep^{-\sigma})^c)}+\norm{\psi}^2_{H^{i+k}(B(\ep^{-\sigma})^c)}\right)
\ee where $C$ depends only on $\bar{m}.$  We have 
\be \label{Tail Energyu Bound 1} 
\norm{\frac{\partial^i}{\partial \tau^i}U}_{H^k(B(c \tau +\ep^{-\sigma})^c)} \leq C\left(\norm{\nabla \phi}_{H^{i+k-1}(B(\ep^{-\sigma})^c)}+\norm{\psi}_{H^{i+k-1}(B(\ep^{-\sigma})^c)}\right).
\ee
It follows from Lemma \ref{Shrinking Tail}, that for any $\sigma>0$
\be \label{Def of Weighted Norm}
\norm{\psi}_{H^{i+k}(B(\ep^{-\sigma})^c)} \leq \ep\norm{\psi}_{H^{i+k}_{\sigma}} \coloneqq \ep\sum_{j=0}^{i+k}\norm{\left(1+\left| \ \cdot \ \right|\right)^{\sigma^{-1}}\left|D^j\psi\right|}_{L^2}.
\ee
Therefore,  as long as $i\geq 1$, \eqref{Tail Energyu Bound 1} becomes 
\be \label{Final Tail Energy Bound}
\norm{\frac{\partial^i}{\partial \tau^i}U}_{H^k(B(c |\tau| +\ep^{-\sigma})^c)} \leq \ep C\left(\norm{\nabla\phi}_{H^{i+k-1}_\sigma}+\norm{\psi}_{H^{i+k-1}_\sigma}\right),
\ee
which holds for all $\tau$.

\section{Residual Estimates}\label{Bounding the Residual}
We bound the terms from left to right in \eqref{Final Form of Res} by the initial conditions given by \eqref{Initial Conditions}.
\begin{lemma}\label{1st term}
For $U$ satisfying \eqref{The efffective equation} with sufficiently smooth initial conditions given by \eqref{Initial Conditions}, there exists a constant $C$ depending only on $\bar{m}$ s.t.  
\be\label{term one}\sup_{|t| \leq \ep^{-1}T}\ep^{-1}\norm{\Delta U(\ep  \cdot ,\ep t)-\epsilon^2 \Delta_\X U(\ep \cdot ,\ep t)}_{\ell^2} \leq \ep^{2}4C\left(\norm{\nabla\phi}_{H^5}+\norm{\psi}_{H^5}\right) \ee
\end{lemma}
\begin{proof}
Recall the definition of $\Delta_\X U$ in \eqref{Macroscopic Laplacian}. For each $i=1,2$, by Taylor's Theorem with remainder, there exists a $\hat{\bj}$ with $\hat{j}_i \in [j_i-1,j_i+1]$ s.t 
\be \label{Taylor's Theorem}
\Delta_i U(\ep \bj)-\ep^2 \partial_{X_iX_i}(U)(\ep \bj) = \ep^4\partial_{X_iX_iX_iX_i}(U)(\ep \hat{\bj}).
\ee
By convexity of $(\cdot)^2$,
\be \label{Common Inequality}
\left(\Delta U(\ep \bj)-\ep^2 \Delta_\x(U)(\ep \bj)\right)^2 \leq 2\sum_{i=1}^{2}\left(\Delta_i U(\ep \bj)-\ep^2 \partial_{X_iX_i}(U)(\ep \bj)\right)^2.
\ee
Combining \eqref{Taylor's Theorem} and \eqref{Common Inequality} yields 
\be 
\left(\Delta U(\ep \bj)-\ep^2 \Delta_\x(U)(\ep \bj)\right)^2 \leq \ep^8 2\sum_{i=1,2}\left(\partial_{X_iX_iX_iX_i}(U)(\ep \hat{\bj})\right)^2. \ee
Summing over $\bj$ and using Corollary \ref{Obvious Corollary} we find 
\be 
\norm{\Delta U(\ep  \cdot ,\ep t)-\epsilon^2 \Delta_\X U(\ep \cdot ,\ep t)}_{\ell^2} \leq \ep^3 4 \norm{D^4U}_{H^2}.
\ee
Now we apply \eqref{Basic Energy Bound 1} to get a constant $C$ which depends only on $\bar{m}$ s.t.
\be 
\norm{\Delta U(\ep  \cdot ,\ep t)-\epsilon^2 \Delta_\X U(\ep \cdot ,\ep t)}_{\ell^2} \leq  \ep^{3}4C\left(\norm{\nabla\phi}_{H^5}+\norm{\psi}_{H^5}\right).
\ee
\end{proof}
\begin{lemma}\label{Second Term}
For $U$ satisfying \eqref{The efffective equation} with sufficiently smooth initial conditions given by \eqref{Initial Conditions}, there exists a constant $C$ depending on $a,b,$ and $\bar{m}$ s.t.
\be 
\sup_{|t| \leq \ep^{-1}T}\ep \norm{z\1_{D(\0,R_\ep)^c}U_{\tau\tau}}_{\ell^2} \leq \ep C\left(\norm{\nabla \phi}_{H^{3}_\sigma}+\norm{\psi}_{H^{3}_\sigma}\right)
\ee 
\end{lemma}
\begin{proof}
 First recall that $z(\bj) \in [a-\bar{m},b-\bar{m}]$. There exists a $C$ that depends on $a,b$ or $\bar{m}$ s.t.  
 \be \norm{z\1_{D(\0,R_\ep)^c}U_{\tau\tau}}_{\ell^2} \leq C\norm{\1_{D(\0,R_\ep)^c}U_{\tau\tau}}_{\ell^2}.\ee
 Recall the definition of $R_\ep$ in \eqref{The Radius}. Note that $R_\ep$ has the form $\ep^{-1}\tilde{R}(T)+1$ where $\tilde{R}(\tau)\coloneqq c\tau+\ep^{-\sigma} $. According to Lemma \ref{Tail Sum}
 \be
 \norm{\1_{D(\0,R_\ep)^c}U_{\tau\tau}}_{\ell^2} \leq 2\ep^{-1}\norm{U_{\tau\tau}}_{H^2(B(\tilde{R}(T))^c)}.
 \ee
 Since $|\tau| \leq T,$
 
 \be 
 \norm{U_{\tau\tau}}_{H^2(B(\tilde{R}(T)^c)} \leq  \norm{U_{\tau\tau}}_{H^2(B(\tilde{R}(\tau)^c))}.
 \ee
 It follows from \eqref{Final Tail Energy Bound} that 
 \be 
 \norm{U_{\tau\tau}}_{H^2(B(\tilde{R}(\tau))^c)} \leq \ep C\left(\norm{\nabla\phi}_{H^{3}_\sigma}+\norm{\psi}_{H^{3}_\sigma}\right).
 \ee
  Stringing these inequalities together and taking $\sup$ yields the result.
\end{proof}
\begin{lemma}\label{Third Term} Let $U$ satisfy \eqref{The efffective equation} with sufficiently smooth initial conditions given by \eqref{Initial Conditions}. There exists a constant $C_\omega$ a.s. s.t.
 \be 
 \sup_{|t| \leq \ep^{-1}T}\ep\norm{\delta_i(\chi)\delta_i^-(U_{\tau\tau})}_{\ell^2} \leq \ep\log(R_\ep) C_{\omega}\sqrt{T+1}\left( \norm{\nabla\phi}_{H_w^4}+\norm{\psi}_{H_w^4}\right) 
 \ee
\end{lemma}
\begin{proof}
Recall the bound given by \eqref{Main Prob Thm 2}. We therefore have that 
\be \label{Unnameable 5}
\norm{\delta_i(\chi_{R_\ep})\delta_i^-(U_{\tau\tau})}_{\ell^2} \leq C_\omega\left(\norm{\log^+(\cdot)\delta_i^{-}\left(U_{\tau\tau}\right)}_{\ell^2} +\log(R_{\ep})\norm{\delta_i^-(U_{\tau\tau})}_{\ell^2} \right).
\ee

Recalling the definition of $w$ in \eqref{Assumption 1 on Weight}
\be 
\log^+(\cdot ) \leq \log(\ep^{-1})\log( \ep \cdot+1) \leq   \log(R_\ep)w_\ep( \cdot  ).
\ee
where 
\be 
w_{\ep}(\cdot) \coloneqq w(\ep\cdot).
\ee
Then
\be 
\norm{\delta_i^-(U_{\tau\tau})}_{\ell^2} \leq \norm{w_\ep\delta_i^-(U_{\tau\tau})}_{\ell^2} \ \text{and} \
 \norm{\log^+(\cdot)\delta_i^{-}\left(U_{\tau\tau}\right)}_{\ell^2} \leq \log(R_\ep)\norm{w_\ep\delta_i^-(U_{\tau\tau})}_{\ell^2}.
\ee
Applying Corollary \ref{Riemann Sum} we have 
\be 
\norm{w_\ep\delta_i^{-}\left(U_{\tau\tau}\right)}_{\ell^2} \leq \ep^{-1}\norm{w(\cdot)\left(U_{\tau\tau}(\cdot) -U_{\tau\tau}(\cdot -\ep \e_i)\right)}_{H^2}.
\ee
Note that $\delta_{\ep i}^-(U)(\cdot, \tau)\coloneqq U(\cdot,\tau)-U(\cdot-\ep \e_i,\tau)$ solves the same wave equation $U$ does with initial conditions 
\be 
\delta_{\ep i}^-U(\X,0)=\delta_{\ep i}^-(\phi)(\X) \ \text{and} \ \delta_{\ep i}^-U_{\tau}(\X,0)=\delta_{\ep i}^-(\phi)(\X).
\ee
Hence we can apply \eqref{Bound on Weighted Energy}
\be \norm{w(\cdot)\left(U_{\tau\tau}(\cdot) -U_{\tau\tau}(\cdot -\ep \e_i)\right)}_{H^2} \leq C\sqrt{|\tau|+1}\left(\norm{\delta^-_{\ep i}(\nabla\phi)}_{H^3_w}+\norm{\delta_{\ep i}^-(\psi)}_{H^3_w}\right).\ee  By Lemma \ref{FT} 
\be 
\norm{w(\cdot)\left(U_{\tau\tau}(\cdot) -U_{\tau\tau}(\cdot -\ep \e_i)\right)}_{H^2} \leq \ep C\sqrt{|\tau|+1}\left(\norm{\nabla\phi}_{H^4_w}+\norm{\psi}_{H^4_w}\right).
\ee 
Stringing the inequalities together and taking $\sup$ gives us the result.
\end{proof}
\begin{lemma}\label{Fourth term}
Let $U$ satisfy \eqref{The efffective equation} with sufficiently smooth initial conditions given by \eqref{Initial Conditions}. There exists a constant $C_\omega$ a.s. s.t.
\be 
\sup_{|t| \leq \ep^{-1}T}\ep \norm{S_i^+(\chi_{R_\ep})\Delta_i(U_{\tau\tau})}_{\ell^2} \leq \ep^{2}R_\ep\log(R_\ep)^{\frac{3}{2}} C_\omega \sqrt{T+1}\left( \norm{\nabla\phi}_{H_w^5}+\norm{\psi}_{H_w^5}\right)
\ee
\end{lemma}
\begin{proof}
The proof is essentially the same as the proof of the previous theorem but now we use \eqref{Main Prob Thm 1} instead of \eqref{Main Prob Thm 2} and Corollary \ref{FTC} instead of Lemma \ref{FT} and then \eqref{Bound on Weighted Energy}.
\end{proof}

\begin{lemma}\label{5th term}
Let $U$ satisfy \eqref{The efffective equation} with sufficiently smooth initial conditions given by \eqref{Initial Conditions}. There exists a constant $C_\omega$ a.s. s.t. 
\be 
\sup_{|t| \leq \ep^{-1}T}\ep^3\norm{m\chi_{R_\ep}U_{\tau\tau\tau\tau}}_{\ell^2} \leq \ep^{2}R_\ep\log(R_\ep)^{\frac{3}{2}} C_\omega \sqrt{T+1}\left( \norm{\nabla \phi}_{H_w^5}+\norm{\psi}_{H_w^5}\right)
\ee
\end{lemma}
\begin{proof}
Since $m(\bj) \in [a,b]$, there exists a constant depending on $a$, $b$ or $\bar{m}$ s.t. 
\be \norm{m\chi_{R_\ep}U_{\tau\tau\tau\tau}}_{\ell^2} \leq C \norm{\chi_{R_\ep}U_{\tau\tau\tau\tau}}_{\ell^2}.\ee The steps are now very similar to those found in Lemma \ref{Third Term} but with \eqref{Main Prob Thm 0} instead of \eqref{Main Prob Thm 2}. We then use Lemma \ref{Riemann Sum} and \eqref{Bound on Weighted Energy}.
 \end{proof}
 
\section{Residual Bound and Discussion}
\label{Main Estimate Section}
\subsection{Proof of Proposition \ref{Res bound prop}}
\begin{proof}
Recall the calculation for $\Res\tilde{u}$ in \eqref{Final Form of Res}. We have bounded each of the  terms, in order, using Lemmas \ref{term one}, \ref{Second Term}, \ref{Third Term}, \ref{Fourth term} and  \ref{5th term}. We obtain Proposition \eqref{Res bound prop} by using the largest parts of each of the bounds.
\end{proof}
Lemma \ref{term one} bounds a completely deterministic term that would appear no matter how the masses are chosen. Lemma \ref{Second Term} bounds a term that arises due to our use of a cut-off function. Recall we need to use a cut-off in order to work around solving \eqref{Chi Easy}. In the case where this can be solved, say for example where the masses vary periodically, then this term would not appear. The norm, $H^3_\sigma$, in this bound is the largest norm. As $\sigma$ becomes smaller, this norm becomes larger. The term in  Lemma \ref{Third Term} is the first term that must be bounded using probabilistic arguments. The constant $C_\omega$ exists almost surely but depends on the actual realization of masses. It therefore could be arbitrarily large, since there is always a small probability that \eqref{Borel-Cantelli} holds only for $n$ extremely large. Thus $C_\omega$ may be worthy of statistical quantification in a follow up work. Recalling the definition of $R_\ep$ in \eqref{The Radius}, the bound for the term in Lemma \eqref{Fourth term} is the dominant one in $\ep$. It is $O(\ep^{1-\sigma}\log(\ep^{-1-\sigma})^{\frac{3}{2}})$. This term also requires the most smoothness and decay of the initial conditions. The final term, bounded in Lemma \eqref{5th term}, is also $O(\ep^{1 -\sigma}\log(\ep^{-1-\sigma})^{\frac{3}{2}})$ for similar reasons. It may be conjectured that $\sigma$ here is artificial as it arises from our inability to solve \eqref{Chi Easy}. We could analyse more; for example, we could find the dependence of $T, \bar{m}, a $ or $b$, but what we are most interested in is tracking $\ep$. 

\subsection{Other Masses}\label{Other Masses}
The method we have employed is robust enough to consider other ways of realizing the masses. For example, we may consider periodic masses by which we mean there exists a positive integer $k$ s.t. $m(j_1,j_2)=m(j_1+k,j_2+k)$. Then $\chi$ is periodic and bounded, so the analysis of the terms appearing in the residual becomes simpler. For instance, we no longer need the term with the cutoff. In this case, one of the largest terms in $\ep$ is the one given in Lemma \ref{Third Term}. A quick count shows that an $\ep$ is lost from converting  the $\ell^2$ norm to an $H$ norm, but an $\ep$ is picked up on account of the finite difference. Thus the term is $O(\ep)$, which produces roughly the same size residual in $\ep$ as we obtained for the i.i.d. masses. The only difference is that for the i.i.d. case, the residual is slightly larger due to the logarithms and the use of the cutoff. 

We also may consider masses which are all identical. In such a case, the only term which appears in the residual is the one bounded in Lemma \ref{1st term}. This improves the accuracy of approximate solution substantially as the residual would be $O(\ep^2)$. 

Another generalization we can make is that masses need not be identically distributed, so long as they all fall into some interval $[a,b] \in \R^+$ and have the same expected value. Since our methods did not use any other features of the masses being identically distributed, e.g. equal variances or probabilities, our result extends to this case without modification. 

Another important example is layered media. Suppose that 
\be  \label{Indepence}
\{m(j_1,j_2)\}_{j_1=-\infty}^{\infty}.
\ee
is random i.i.d. sequence of masses and that for all $j_1$
\be \label{all equal}
\cdots =m(j_1,-1)= m(j_1,0)=m(j_1,1)=\cdots 
\ee
In this case, it is actually possible to solve \eqref{Chi Easy} and thereby not use a cutoff, but we proceed using the same tools we have developed, since such tools can be utilized in higher dimensions. In this case, Hoeffding's Inequality does not immediately hold since we do not have complete independence. Reconsider \eqref{Reconsider}
\be \begin{aligned}
\mathcal{L}(\chi_{R_\ep})(j_1,j_2)&=\sum_{|k_1|\leq r}\sum_{|k_2|\leq r}\mathcal{L}(\varphi)(j_1-k_1,j_2-k_2)z(k_1,k_2)\\&=\sum_{k_1=j_1-r}^{j_1+r}\sum_{k_2=j_2-r}^{j_2+r}\mathcal{L}(\varphi)(k_1,k_2)z(j_1-k_1,j_2-k_2).
\end{aligned}
\ee
Let $z(k_1)\coloneqq z(k_1,\cdot).$ This is well defined because of \eqref{all equal}. Thus 
\be 
\mathcal{L}(\chi_{r})(j_1,j_2) =\sum_{k_1=j_1-r}^{j_1+r}z(j_1-k_1)\sum_{k_2=j_2-r}^{j_2+r}\mathcal{L}(\varphi)(k_1,k_2).
\ee
Let 
\be 
\mathcal{L}(\varPhi_r)(j_2,k_1)\coloneqq \sum_{k_2=j_2-r}^{j_2+r}\mathcal{L}(\varphi)(k_1,k_2)
\ee
so 
\be 
\mathcal{L}(\chi_{R_\ep})(j_1,j_2)=\sum_{k_1=j_1-r}^{j_1+r}z(j_1-k_1)\mathcal{L}(\varPhi_r)(j_2,k_1).
\ee
This is a sum involving independent mean zero random variables which are contained in some interval, so we may apply Hoeffding's Inequality. Let 
\be \label{Weird Norm}
\norm{\mathcal{L}(\varPhi_r)(j_1,j_2)}^2\coloneqq \sum_{k_1=j_1-r}^{j_1+r}\mathcal{L}(\varPhi_r)(j_2,k_1)^2.
\ee
Then we have by Hoeffding's Inequality that 
\be
P\left(|\mathcal{L}(\chi_r)(j_1,j_2)| \geq t\right) \leq 2P\left(\frac{-2t^2}{(a-b)^2\norm{\mathcal{L}(\varPhi_r)(j_1,j_2)}^2}\right).
\ee
Now the same argument works as is made in Section \ref{Estimates on the Solution} but the relevant quantity to calculate is the square root of \eqref{Weird Norm}. Take $\mathcal{L}$ to be any of the operators in Section \ref{Estimates on the Solution}. Then 
\be 
\norm{\mathcal{L}(\varPhi_r)(j_1,j_2)}^2=\sum_{k_1=j_1-r}^{j_1+r}\left(\sum_{k_2=j_2-r}^{j_2+r}\mathcal{L}(\varphi)(k_1,k_2)\right)^2.
\ee
It is possible to use Jensen's inequality to obtain 
\be 
\norm{\mathcal{L}(\varPhi_r)(j_1,j_2)}^2\leq(2r+1)\sum_{k_1=j_1-r}^{j_1+r}\sum_{k_2=j_2-r}^{j_2+r}\mathcal{L}(\varphi)(k_1,k_2)^2=(2r+1)\norm{\mathcal{L}\varphi}^2_{D(\bj,r)}.
\ee
One notices that $\norm{\mathcal{L}\varphi}^2_{D(\bj,r)}$ are norms we have already computed in Section \ref{Estimates on the Solution}, see \eqref{Norm of phi} for example. The $r$ out front provides an extra half power of $\ep^{-1-\sigma}$ when one considers \eqref{The Radius} and after taking square roots as in \eqref{Almost Sure Chi}. Therefore, in contrast with \eqref{Res Bound}, we have
\be\label{half power worse} 
\sup_{|t| \leq \ep^{-1}T}\norm{\Res \tilde{u}}_{\ell^2} \leq \ep^{\frac{1}{2}-\frac{3}{2}\sigma}C_\omega\log(\ep^{-1-\sigma})^{\frac{3}{2}}\norm{\phi,\psi}_{H^{5}_\sigma}.
\ee
This example sheds some light on the complexity introduced when considering random masses. In contrast, for periodic masses, even if \eqref{all equal} holds, then \eqref{Chi Easy} is solvable and the solution is periodic so most importantly bounded. Therefore the residual is always $O(\ep).$

\subsection{Main Estimate Result} 
\begin{theorem}\label{Main Estimate}
Suppose $\psi, \phi \in H_\sigma^{5}$. Let $u$ satisfy \eqref{Summation by Parts} with 
\be 
u(\bj ,0)=\ep^{-1}\phi(\ep \bj) \ \text{and} \
\dot{u}(\bj,0)=\psi(\ep \bj).
\ee
and let $U$ satisfy \eqref{The efffective equation} with 
\be 
U(\X,0)=\phi(\X) \ \text{and} \ \partial_{\tau} U(\X,0)= \psi(\X).
\ee
With
\be 
\hat{u}(\bj,t)=\ep^{-1}U(\ep \bj,\ep t)
\ee 
where $\chi_{R_\ep}$ is given by \eqref{Chi restricted} and $R_\ep$ by \eqref{The Radius}.  Then there exists a $C_\omega(a,b,\bar{m},T)$ a.s. s.t.
\be \label{Main Bound 1} 
\sup_{|t| \leq \ep^{-1}T}\norm{u-\hat{u}}_{\ell^2}\leq \ep^{-1-\sigma}C_\omega \log(\ep^{-1-\sigma})^{\frac{3}{2}}\norm{\phi,\psi}_{H^5_\sigma}. 
\ee
and 
\be \label{Main Bound 2}
\sup_{|t| \leq \ep^{-1}T}\norm{\dot{u}-\dot{\hat{u}}}_{\ell^2}\leq \ep^{-\sigma}C_\omega \log(\ep^{-1-\sigma})^{\frac{3}{2}}\norm{\phi,\psi}_{H^5_\sigma}. 
\ee
\end{theorem}

\begin{remark} \label{Relative Error}
Although the right hand side of \eqref{Main Bound 1} and \eqref{Main Bound 2} appear to grow, the size of $u$ and $\dot{u}$ is roughly $\ep^{-2}$ and $\ep^{-1}$, so the error is relatively small. 
\end{remark}
\begin{proof}
Let 
\be 
\tilde{u}(\bj,t)=\ep^{-1}U(\ep \bj, \ep t)+\ep\chi_{R_\ep}(\bj)U_{\tau\tau}(\ep \bj, \ep t). 
\ee
This is what has typically been our $\tilde{u}$ as defined in \eqref{Approximate Solution 2}. Since we have taken the initial conditions for $u$ and $\hat{u}$ to be identical, \eqref{Assumption on Disparity of Initial Conditions} holds as long as we have
\be 
\ep\norm{\ep \chi_{R_\ep}(\cdot)U_{\tau\tau\tau}(\cdot,0),\delta_1^+\left(\chi_{R_\ep}U_{\tau\tau}\right)(\cdot,0),\delta_2^+\left(\chi_{R_\ep}U_{\tau\tau}\right)(\cdot,0)} \leq \ep^{-1}C\sup_{|t| \leq \ep^{-1}T}\norm{\Res{\tilde{u}}}_{\ell^2}
\ee
which can be proven using the same kind of arguments given in Section \ref{Bounding the Residual}.

Then we have by \eqref{Inequality for Microstate} and \eqref{Inequality for u} that for a constant depending on $a,b$, and $T$ s.t.
\be 
\sup_{|t| \leq \ep^{-1}T}\norm{\dot{u}-\dot{\tilde{u}}}_{\ell^2} \leq \ep^{-1}C\sup_{|t| \leq \ep^{-1}T}\norm{\Res{\tilde{u}}}_{\ell^2}
\ee
and 
\be
\sup_{|t| \leq \ep^{-1}T}\norm{u-\tilde{u}}_{\ell^2} \leq \ep^{-2}C\sup_{|t| \leq \ep^{-1}T}\norm{\Res{\tilde{u}}}_{\ell^2}.
\ee
Now we use \eqref{Res Bound} to obtain the correct right hand side in \eqref{Main Bound 1} and \eqref{Main Bound 2} but for $\tilde{u}$ instead of $\hat{u}.$ Thus it remains to analyse the difference between the two which is
\be 
\sup_{|t| \leq \ep^{-1}T}\norm{\dot{\hat{u}}-\dot{\tilde{u}}}_{\ell^2} \leq \sup_{|t| \leq \ep^{-1}T}\norm{\ep^2\chi_{R_\ep}(\cdot)U_{\tau \tau\tau}(\ep \codt)}_{\ell^2}
\ee
and 
\be 
\sup_{|t| \leq \ep^{-1}T}\norm{\hat{u}-\tilde{u}}_{\ell^2} \leq \sup_{|t| \leq \ep^{-1}T}\norm{\ep\chi_{R_\ep}(\cdot)U_{\tau \tau}(\ep \codt)}_{\ell^2}.
\ee
Both of these can be bounded using steps similar to those in Lemmas \ref{Third Term} and \ref{Fourth term}. One finds
\be 
\sup_{|t| \leq \ep^{-1}T}\norm{\dot{\hat{u}}-\dot{\tilde{u}}}_{\ell^2} \leq \ep C_{\omega}R_\ep\left((\sqrt{T}+\log(R_\ep)^{\frac{3}{2}})( \norm{\phi}_{H_w^3}+\ep\norm{\psi}_{H_w^2})\right)
\ee
\be 
\sup_{|t| \leq \ep^{-1}T}\norm{\hat{u}-\tilde{u}}_{\ell^2} \leq C_{\omega}R_\ep\left((\sqrt{T}+\log(R_\ep)^{\frac{3}{2}})( \norm{\phi}_{H_w^2}+\ep\norm{\psi}_{H_w^1})\right).
\ee
Recalling the definition of $R_\ep$ in \eqref{The Radius}, these bounds have the correct power of $\ep$. Thus, with the appropriate constant $C_\omega$, everything can be dominated by the right hand side in \eqref{Main Bound 1} and \eqref{Main Bound 2}.
\end{proof}
\section{Coarse Graining}\label{Coarse Graining}

Theorem \ref{Main Estimate} says that the macroscopic behavior of the system is wavelike i.e. it evolves according to an effective wave equation. We formalize this notion by proving a convergence result in the macroscopic setting. We have a number of operators to introduce. The lattice Fourier transform, $F:\ell(\Z^2) \to L^2(\R^2)$, is given by
\be \label{Discrete FT}
F[f](\y)=\frac{1}{(2\pi)^2}\sum_{\bj \in \Z^2}\exp\left(-i\bj\cdot\y \right)f(\bj).
\ee
Here $y \in \R^2$. Its inverse is
\be \label{Discrete FT Inverse}
F^{-1}[g](\bj)=\int_{[-\pi,\pi]^2}\exp\left(i\y \cdot \bj \right)g(\y)d\y. \ee Let $\mathcal{F}:
L^2(\R^2) \to L^2(\R^2)$ be the typical Fourier transform and $\mathcal{F}^{-1}$ be its inverse. The sampling operator is 
\be \label{Sampling}
\mathcal{S}(u)(\bj)=u( \bj).
\ee
A cutoff operator is 
\be 
\theta_{\phi}(\y)=\begin{cases}
1 & \y \in [-\phi, \phi]^2 \\ 
0 & \text{else}
\end{cases}.
\ee
Finally, we define a low pass interpolator. For a continuous variable $\x \in \R^2$
\be \label{low pass}
\mathcal{L}[f](\x)=\mathcal{F}^{-1}[\theta_{\pi}(\cdot)F[f](\cdot)](\x).
\ee
 The following theorem says that the abstract diagram found in \cite{Mielke} holds in the setting of i.i.d. masses almost surely. The diagram depicts how time evolution commutes with the coarse-graining operator $\mathcal{L}$ in the limit as $\ep \to 0$, meaning that one can first evolve the system according to the \eqref{Main2}, and then apply $\mathcal{L}$, or one can apply $\mathcal{L}$ initially to obtain macroscopic initial conditions and then evolve those according the effective wave equation and arrive at the same result.
 \subsection{Main Convergence Result} \label{Main Convergence Result}
 \begin{theorem}\label{Main Coarse Graining}
 Let $U$ solve \eqref{The efffective equation} with initial conditions given by \eqref{Macroscopic Initial Conditions}. 
 Let $u$ solve \eqref{Main2} with initial conditions given by \eqref{Initial Conditions}.
 Put 
 \be
 U_\ep(\X,\tau)=\ep\mathcal{L}(u(\cdot,\tau/\ep))(\X/\ep) 
 \ee
 so that 
 \be 
 \partial_{\tau}U_\ep(\X,\tau)=\mathcal{L}(\dot{u}(\cdot,\tau/\ep))(\X/\ep). \ee 
 Then 
 \be \lim_{\ep \to 0}
 \sup_{|\tau| \leq T}\norm{U_\ep-U,\partial_{\tau} U_\ep-\partial_{\tau}U}_{L^2} =0  
 \ee
almost surely.
 \end{theorem}.
 \begin{remark}
 The rate of convergence is no worse than $\ep^{1-\sigma}\log(\ep^{-1-\sigma})^{\frac{3}{2}}$ as seen in the proof.
 \end{remark}
 \begin{proof}
 
 \be 
 \norm{U_\ep-U}_{L^2}\leq \norm{\ep\mathcal{L}(u(\cdot,\tau/\ep))(\cdot/\ep) -\mathcal{LS}(U(\ep \cdot,\tau))(\cdot/\ep)}_{L^2}+\norm{\mathcal{LS}(U(\ep \cdot,\tau))(\cdot/\ep)-U(\cdot,\tau)}_{L^2}. 
 \ee
 According to Lemma \ref{Interpolation}, the second term goes to $0$ uniformly in $\tau$ at a rate of $\ep$. From the first term, we have
 \be
 \norm{\ep\mathcal{L}(u(\cdot,\tau/\ep))(\codt/\ep) -\mathcal{LS}U(\ep\cdot,\tau)(\cdot/\ep) }_{L^2} =\ep \norm{\ep\mathcal{L}(u(\cdot,\tau/\ep))(\cdot) -\mathcal{LS}U(\ep\cdot,\tau)(\cdot) }_{L^2}.
 \ee
 By Plancherel 
 \be 
 \ep \norm{\ep\mathcal{L}(u(\cdot,\tau/\ep))(\X) -\mathcal{LS}U(\ep\cdot,\tau)(\cdot) }_{L^2} \leq \ep C \norm{\ep u(\cdot,\tau/\ep) -\mathcal{S}U(\ep\cdot,\tau) }_{\ell^2}.
 \ee
 Now we use that 
 $U(\ep \cdot, \tau)=\ep \hat{u}(\cdot, \ep t)$ as in Theorem \ref{Main Estimate} and the result follows from that same theorem. Recall $\tau=\ep t$. Similarly 
 \be 
 \ep \norm{\mathcal{L}(\dot{u}(\cdot,\tau/\ep))(\X) -\mathcal{LS}U_{\tau}(\ep\cdot,\tau)(\cdot) }_{L^2} \leq \ep \norm{ \dot{u}(\cdot,\tau/\ep) -\mathcal{S}U_{\tau}(\ep\cdot,\tau) }_{\ell^2}. 
 \ee
 Again we use Theorem \ref{Main Estimate} where $U_\tau(\ep \cdot, \tau)=\dot{\hat{u}}(\cdot ,t).$
 \end{proof}
 \section{Numerical Results}
 \label{Numerical Results}
Our numerical results focus on confirming the upper bounds found in Theorem \ref{Main Estimate}. We refer to the left hand side of \eqref{Main Bound 1} as the absolute error of the displacement (a.e.d.) and the left hand side of \eqref{Main Bound 2} as the absolute error of the velocity (a.e.v). For the next two experiments, we have chosen
\be 
\phi(X_1,X_2)=\sech\left(\frac{1}{2}(X_1-1)^2+(X_2-1)^2\right), \ \psi(X_1,X_2)= \sech\left((X_1+1)^2+\frac{1}{2}(X_2+1)^2\right)
\ee
and looked at $\ep$ over $\{1/2,1/4,1/8,1/16,1/32\}.$
Every $m(j)$ is sampled from $\{1/2,3/2\}$ and for the first experiment the masses are chosen to be i.i.d. As $\epsilon$ varies, this grid of randomly chosen masses remains fixed.
\begin{figure}
   \begin{subfigure}[t]{.49\textwidth}
    \centering
        \includegraphics[width=\linewidth, height=2.5in]{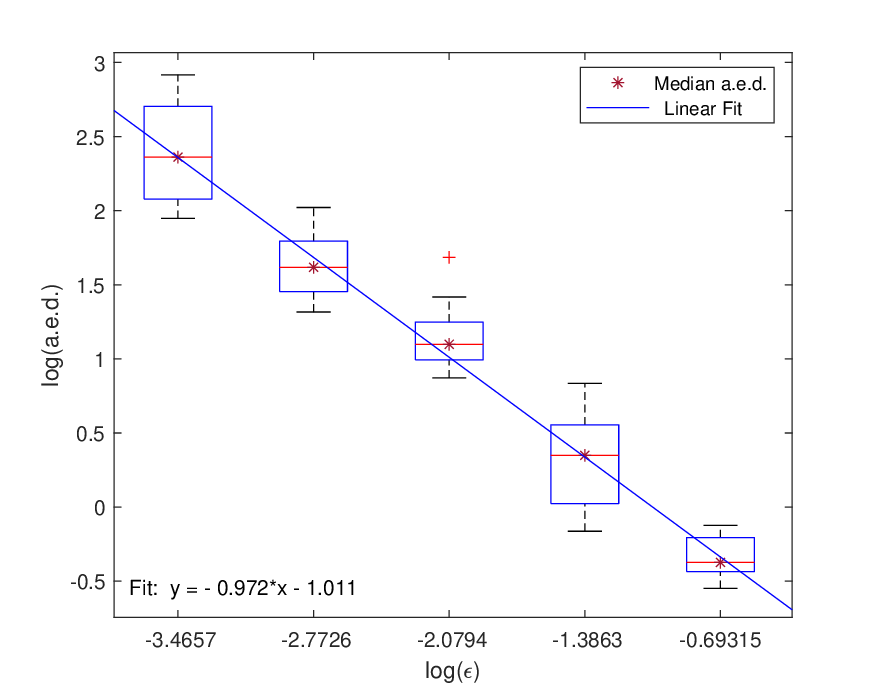}
        \end{subfigure}
        ~
        \begin{subfigure}[t]{.49\textwidth}
        \centering
        \includegraphics[width=\linewidth,height=2.5in]{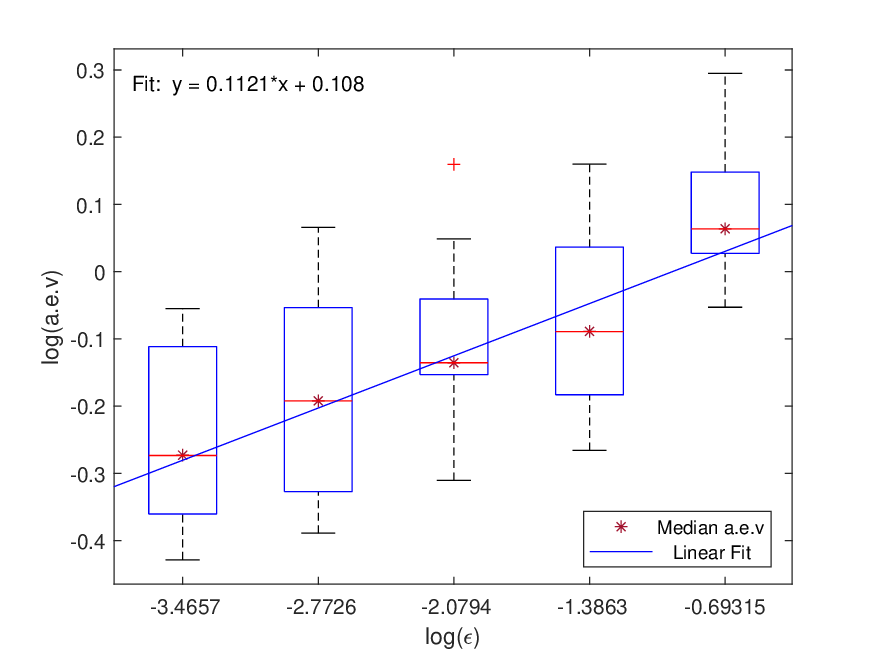}
        
        \end{subfigure}
    \caption{The left (right) panel shows a $\log$-$\log$ graph of a.e.d. (a.e.v) versus $\ep$. Since the masses are chosen randomly, ten realization of masses were tested. The distribution of the error is plotted with a box and whisker plot at each value of $\ep$. The slope of the line of best fit numerically approximates the power of $\ep$ in \eqref{Main Bound 1} (\eqref{Main Bound 2}). }
    \label{All Random Experiment}
    \end{figure}
The upper bound for the a.e.v. obtained in Theorem \ref{Main Estimate} is $\ep^{-\sigma}\log(\ep^{-1-\sigma})^{\frac{3}{2}}$ for the a.e.v. (Recall that $\sigma$, is any arbitrarily small positive number.) The slopes seen in Figure \ref{All Random Experiment} are in  agreement with the bounds obtained in the theorem, i.e. the slopes reflect to what power of $\ep$ the error depends. In fact the estimate is close to sharp. We can thus think of such bounds as giving an analytic prediction on the size of the error in many cases. 

For the second experiment, we use the setup for the masses given by \eqref{Indepence} and \eqref{all equal} i.e. the masses are layered. According to \eqref{half power worse}, we expect the slopes to be no more than a half power less than those seen in Figure \ref{All Random Experiment}. This is indeed what we see in Figure \ref{Half Random Experiment}. Again, the numerical error is close to the predicted error.
\begin{figure}
   \begin{subfigure}[t]{.49\textwidth}
    \centering
        \includegraphics[width=\linewidth, height=2.5in]{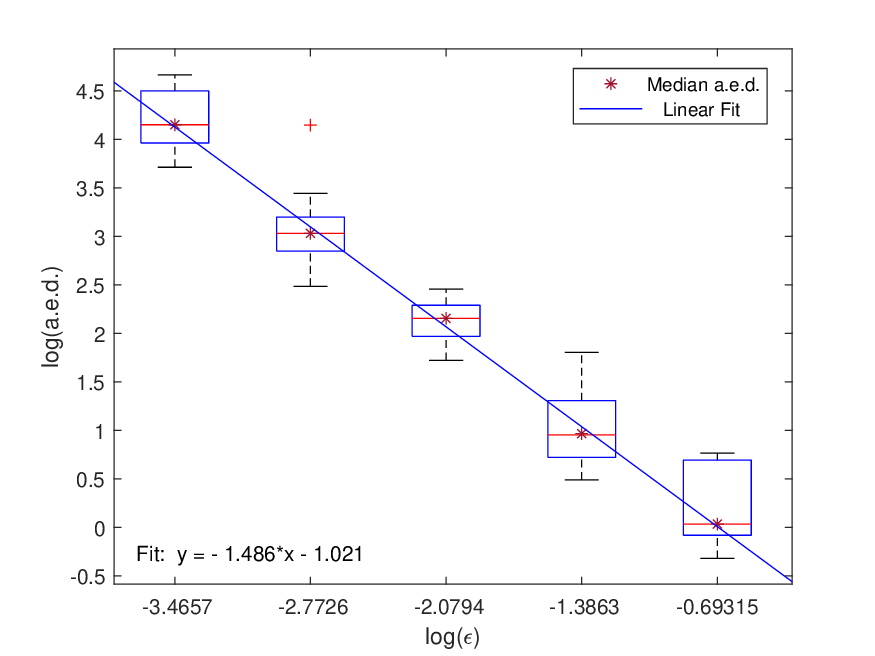}
        \end{subfigure}
        ~
        \begin{subfigure}[t]{.49\textwidth}
        \centering
        \includegraphics[width=\linewidth, height=2.5in]{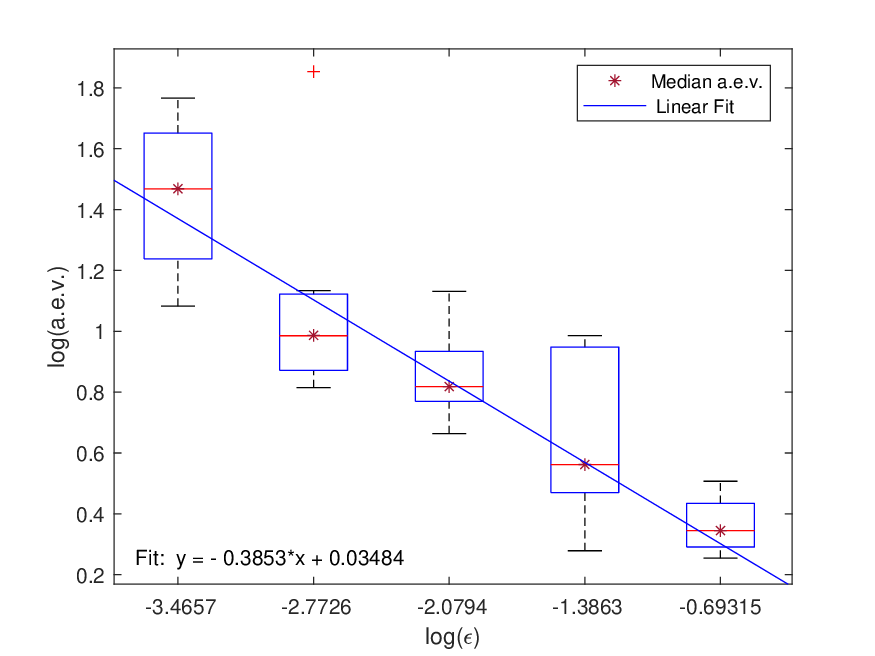}
        \end{subfigure}
    \caption{These panels reflect the same measurements as those in Figure \ref{All Random Experiment}, but now the masses have been chosen according to \eqref{Indepence} and \eqref{all equal}}
    \label{Half Random Experiment}
    \end{figure}
    
    Finally we compare these results to what happens in three deterministic cases. In one case we choose the masses to be constant in which case we would expect the slope for the a.e.v and the a.e.d. to be no worse than $1$ and 0 respectively.  In a second, the masses are periodically layered i.e. they only vary periodically with a period of 2 along one of the coordinate axis and along the other, \eqref{all equal} holds. For the third case, masses are chosen to vary periodically along both coordinate axes with period of 2.  In both cases we expect the slope of the a.e.v. and a.e.d. to be no worse than $0$ and $1$. In both cases the upper bound holds; however, unlike in all the previous experiments, numerically computed a.e.d. seems to be substantially better than the analytic upper bound, since the slope for both periodic cases is closer to $0$ than to $1$. 
    
\begin{figure}
   \begin{subfigure}[t]{.49\textwidth}
    \centering
        \includegraphics[width=\linewidth,height=2.5in]{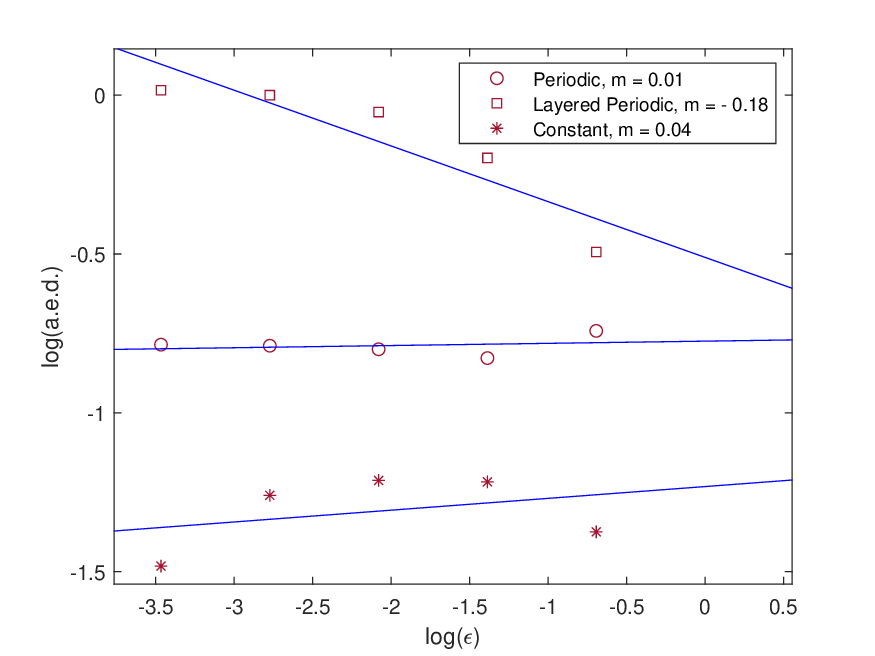}
        \end{subfigure}
        ~
        \begin{subfigure}[t]{.49\textwidth}
        \centering
        \includegraphics[width=\linewidth, height=2.5in]{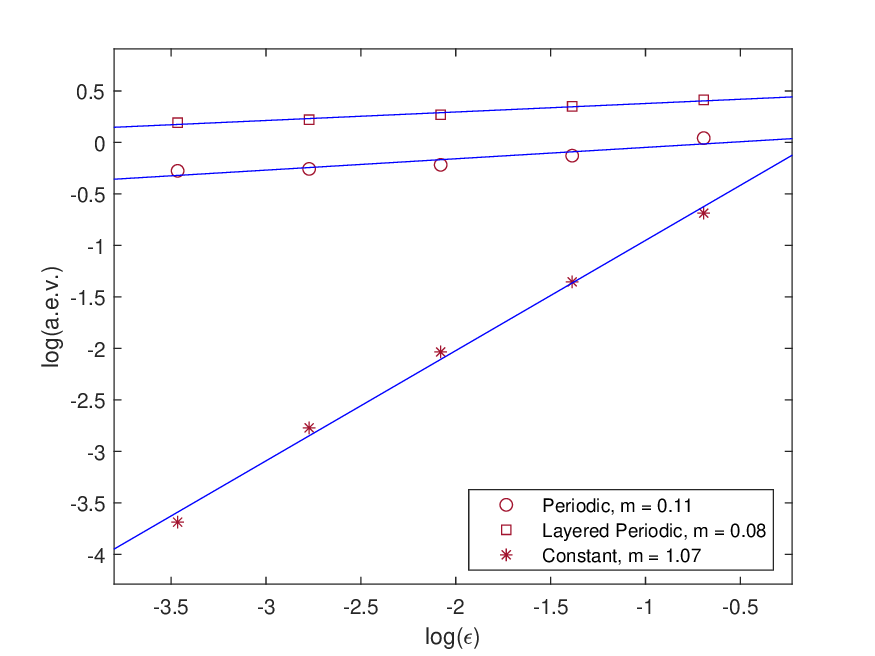}
        \end{subfigure}
    \caption{Masses have been chosen deterministically in three different ways.}
    \label{Deterministic}
    \end{figure}
    
 One important observation is that when the randomly chosen masses are layered, the approximation to the wave equation does not converge as quickly as it does when they are not layered. The main physical explanation we propose for the difference in the observed slope between Figures \ref{All Random Experiment} and \ref{Half Random Experiment} is that in the second experiment, reflections caused by the random masses manifest as long ripples transverse to the direction in which the masses are randomly changing. This is in contrast to the first experiment where reflections manifest as localized disturbances. Figure \eqref{Comparison Figure} gives some empirical evidence for this phenomenon. 
 
One possible heuristic explanation for why we don't get improvement for periodic masses is that there is always a direction in $\R^2$ along which the averages of masses in lines transverse to that direction are varying (unless the masses are all constant).This produces a kind of layering that cannot occur if all the masses are chosen to be i.i.d. since masses along any line average to the same value. Hence, in this sense, the homogenization is more uniform.

\begin{figure}
   \begin{subfigure}[t]{.49\textwidth}
    \centering
        \includegraphics[width=\linewidth]{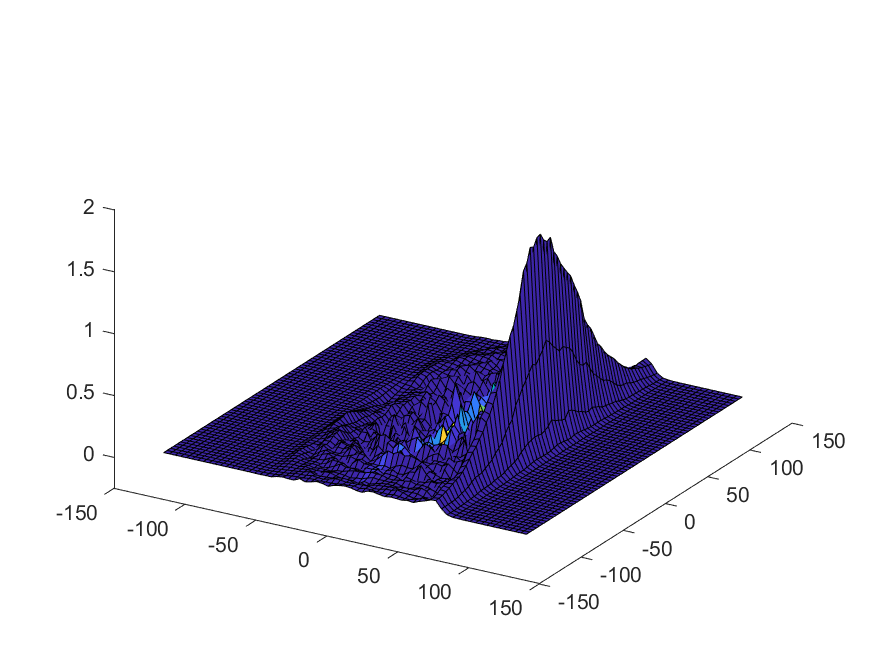}
        \end{subfigure}
        ~
        \begin{subfigure}[t]{.49\textwidth}
        \centering
        \includegraphics[width=\linewidth]{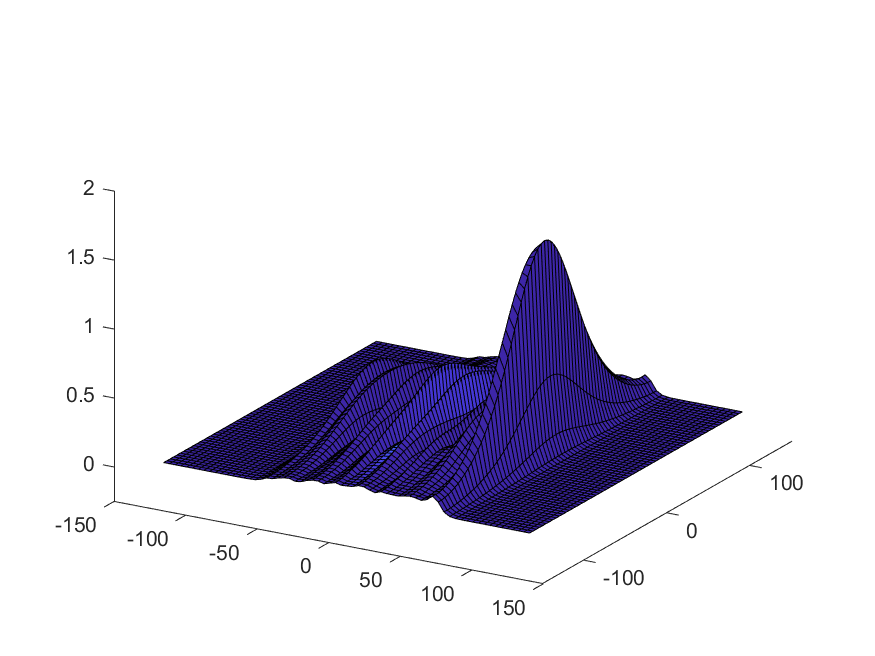}
        \end{subfigure}
    \caption{The left panel shows a wave traveling through masses which are i.i.d. and the right shows a wave traveling through layered random masses. One can spot the long transverse ripples  in the right panel, whereas in the left panel, deviations appear more localized.}
    \label{Comparison Figure}
    \end{figure}

  Finally, we have performed a number of similar experiments in 1D, the results of which are summarized in Table \ref{Comprehensive Table}. The predicted values can all be proven using essentially the same method as what has been demonstrated or discussed in Section \ref{Other Masses} and throughout the rest of the paper. Even though our predicted values are upper bounds, we see that in most cases our estimates are close to sharp. The exception is when the masses are chosen periodically, where the predicted a.e.d. often overshoots by close to a full power of $\ep$. This indicates that integrating $p$, and then applying a triangle inequality to obtain an upper bound on $u$ is not efficient in the periodic case. 
  
  For both the constant coefficient and periodic cases, a half power of $\ep$ is lost with each increase in dimension. This is due to the length scaling and is negated when one considers the coarse-graining limit. We see no reason this trend should not continue into higher dimensions. On the other hand, there is not a decrease in the power of $\epsilon$ in the i.i.d. case. Heuristically, one can see this by comparing the variance of $\chi$ in 1D and in 2D. The fundamental solution $\varphi$ of $\Delta$ plays an important role in the growth rate. In 1D, this fundamental solution is given by $\varphi(\cdot)=\frac{1}{2}|\cdot|.$ Taking the expectation of $\chi_r^2$ and then taking the square-root, where $\chi_r$ is defined by \eqref{Chi restricted}, yields that $\chi_r$ is approximately the size of $r^\frac{3}{2}$ in $\ell^{\infty}$. The same procedure in 2D yields $r\log(r)$.  This accounts for an additional half power loss $\epsilon$ in the 1D case. The argument using the ideas of sub-Gaussian random variables is one way to formalize this intuition and obtain an almost sure estimate. Ultimately, this half power is negated by the increase in dimension, and since the random term in the residual is still the dominant one, we find that the size of the absolute error (as dependent upon $\ep$) roughly does not change from 1D to 2D. Therefore the coarse-graining limit converges faster in 2D.
  
  \begin{table}
 \begin{center}
 \begin{tabular}{|c|c|c|c|c|}
 \hline
 \multicolumn{1}{|c|}{Rates for:}&
\multicolumn{2}{|c|}{a.e.v.}  & \multicolumn{2}{|c|}{a.e.d.}
\\\hline
Mass/Dim. & Predicted & Simulated & Predicted & Simulated
\\
\hline
     1D Const. Coeff. & $1.5$&  $1.5$& $0.5$& $0.5$\\
     \hline 
     1D Period. Coeff. & $0.5$ &$0.7$&$-0.5$ &$0.5$\\
     \hline 
     1D Rando. Coeff. & $-\sigma$ & $0.2$&$-1-\sigma$ &$-0.8$ \\
     \hline 2D Const. Coeff & $1$ & $1.1$ &$0$ & $0.0$  \\
     \hline 2D Period. Coeff & $0$ & $0.1$ &$-1 $&$0.0$ \\
     \hline 2D Rando. Coeff & $-\sigma$& $0.1$  &$-1-\sigma$ & $-1.0$ \\ 
     \hline 2D Layered Coeff & $-0.5 -\sigma$& $-0.4$ &$-1.5-\sigma$ & $-1.5$ \\ \hline
     
 \end{tabular}
 \caption{\label{Comprehensive Table}This table summarizes the epsilon dependence of the absolute errors. The numbers indicate the power of $\epsilon$. $\sigma$ is a small and positive constant. Results have been rounded. }
  \end{center}
 \end{table}
 \section{Conclusion}
 We have rigorously justified the wave equation as descriptor for the macroscopic linearized dynamics of a 2D square lattice composed of i.i.d. masses. We have given analytic as well as numerical evidence that this description is more accurate in 2D than it is in 1D. We conjecture that in 3D, the a.e.d. and a.e.v. is roughly the same as it is in 2D for the i.i.d. case. This means it would be roughly only a half power of epsilon larger than in the constant case. We think such a result is modest evidence that waves propagate more easily in a disordered lattice in higher dimensions. An important exception that was seen to this occurs in the case of layered random masses, where the error became larger from 2D to 1D in the same way it did for periodic or constant masses. Although there are results in the continuous setting that are similar to ours \cite{Papanicolau}, as far as we can tell, this is the first result that provides a rigorous almost sure bounds on the rate of convergence for lattices, and we think that these rates of convergence provide insight into the effects of dimensionality of the disorder on wave propagation. Finally, the techniques introduced, especially the use of sub-Gaussian random variables can probably be used to access similar results for various other discrete systems with different kinds of disorder.
 
\section{Appendix}
\label{Appendix}
\subsection{Probabilistic Estimates A}
\begin{lemma}
\label{Bijection Existence}
There exists a bijection
\be I: \Z^2 \times \N^+/\{1\} \to \N^+\ee
such that 
\be 
I(\bj, r) \leq C\max\{\bj^3,r^3 \}.
\ee 
\end{lemma}
\begin{proof}
Notice the set 
\be 
B(n)\coloneqq  \{(\bj,r) \ | \ |\bj| \leq n \ , \ r \leq n \}
\ee
is a subset of the top half of a ball in $\Z^3.$ If we require
		\be 
		I(\bj_1,r_1) <I(\bj_2,r_2)
		\ee
		whenever there exists an integer $n$ s.t. 
		\be (\bj_1,r_1) \in B(n) \ \ \text{but} \ \ (\bj_2,r_2) \notin B(n) \ee 
		we have at worst that 
		\be 
		I(\bj,r) \leq |B(\max\{\bj,r\})| \leq C\max\{\bj^3,r^3 \}.
		\ee
	
\end{proof}
\begin{lemma}
\label{Log Subadditive}
For $r\geq 2$, and $\bj \in\Z^2,$ there exists a constant $C$ s.t. 
\be\log^+(|\bj|+r) \leq C(\log^+(|\bj|)+\log^+(r)) \ee
\end{lemma}
\begin{proof}
The inequality is trivial for $|\bj|=0$. For $|\bj| \neq 0$, we can prove the inequality for $\log$. By the concavity of $\log$, $\log(\cdot +1)$ is sub-additive.
\be 
\log(|\bj|+r-1+1) \leq \log(|\bj|+1)+\log(r) \leq C( \log(|\bj|)+\log(r)).
\ee
\end{proof}
\begin{lemma}
\label{Bound on delta phi}
For $\varphi$ given by \eqref{Fundamental Solution Explicit} and $\delta_i$ defined in \eqref{Partial Center Difference} we have for some $C$ 
\be |\delta_i(\varphi)(\bj)| \leq \frac{C}{|\bj|+1} \ee
\end{lemma}
\begin{proof}
By \eqref{Fundamental Solution Explicit}, we have
\be 
\delta_i(\varphi)(\bj) =\frac{1}{2\pi}\log^+|\bj+\e_i|-\frac{1}{2\pi}\log^+|\bj-\e_i| +O(|\bj|^{-2})
\ee
Note that when $|\bj|=0$, we are left with only small $O(|j|^{-2})$ terms. For $|\bj|=1$ we have 
\be 
\left|\log^+|\bj+\e_i|-\log^+|\bj-\e_i|\right| \leq \log(2).
\ee
For $|\bj| >1$, we have
\be \begin{aligned}
\left|\log|\bj+\e_i|-\log|\bj-\e_i|\right| &\leq  \left|\log\left(\frac{|\bj|+|\e_i|}{|\bj|-|\e_i|}\right)\right|  \\&=\left|\log\left(\frac{|\bj|-1+2}{|\bj|-1}\right)\right| \\ &\leq \left|\log\left(1+\frac{2}{|\bj|-1}\right)\right| \\ &\leq \left|\log\left(1+\frac{2}{|\bj|+1}\frac{|\bj|+1}{|\bj|-1}\right)\right| \\ &\leq \left|\log\left(1+\frac{4}{|\bj|+1}\right)\right| \\&\leq \frac{4}{|\bj|+1}.
\end{aligned}
\ee
Thus we obtain the result.
\end{proof}
\begin{lemma}
\label{Bound on norm delta phi}
For $\varphi$ given by \eqref{Fundamental Solution Explicit} and $\delta_i$ defined in \eqref{Partial Center Difference}, then when $r \geq 2$, we have for some $C$ 
\be\norm{\delta_i\varphi}^2_{D(\bj,r)} \leq C\log(|\bj|+r) \ee
\end{lemma}
\begin{proof}
By Lemma \ref{Bound on delta phi}
\be \norm{\delta_i\varphi}^2_{D(\bj,r)} \leq C\sum_{\bk \in D(\bj,r)}\left(\frac{1}{|\bk|+1}\right)^2.\ee
The largest magnitude of $|\bk|$ in $D(\bj,r)$ is $|\bj|+r$. Also note that there is some constant $C$ s.t. the number of elements in $D(\bj,r)$ of some magnitude $r'$ is less than $C(r'+1)$. Thus there exists a constant $C$ s.t.  
\be 
\begin{aligned}
\sum_{\bk \in D(\bj,r)}\left(\frac{1}{|\bk|+1}\right)^2&\leq C\sum_{r'=0}^{|\bj|+r}(r'+1)\left(\frac{1}{r'+1}\right)^2 = C\sum_{r'=0}^{|\bj|+r}\left(\frac{1}{r'+1}\right)
\end{aligned}
\ee
A common bound on the harmonic series yields
\be
\sum_{\bk \in D(\bj,r)}\left(\frac{1}{|\bk|+1}\right)^2 \leq C\log(|\bj|+r+2),
\ee
which yields the result.
\end{proof}
\subsection{Lattice Energy Argument A}

\begin{lemma} \label{Product Rule}
Let $f,g:\Z^2 \to \R$. Recall the definition of the operators in \eqref{Partial Center Difference}, \eqref{Partial differences}, \eqref{Shifts}, and \eqref{Partial Laplacian}. Then
\be
\Delta_i (fg)= \Delta_i (f)g+ S_i^+(f)\Delta_i(g)+\delta_i(f)\delta_i^-(g).
\ee
\end{lemma}
\begin{proof}
Starting from the definition in \eqref{Partial Laplacian} we have
\be\begin{aligned}
\Delta_i(fg)=&S_i^+(fg)-2fg+S_i^-(fg) \\
=&S_i^+(f)g-2fg+S_i^-(f)g+S_i^+(fg)+S_i^-(fg) -S_i^+(f)g-S_i^-(f)g \\
=&\Delta_i(f)g+S_i^+(fg) -S_i^+(f)g+S_i^-(fg)-S_i^-(f)g\\
=&\Delta_i(f)g+S_i^+(fg)-2S_i^+(f)g +S_i^+(f)S_i^-(g) \\&+S_i^+(f)g-S_i^+(f)S_i^-(g)  +S_i^-(fg) -S_i^-(f)g \\
=&\Delta_i(f)g+S_i^+(f)\Delta_i(g)+S_i^+(f)(g-S_i^-(g))-S_i^-(f)(g-S_i^-(g))\\
=&\Delta_i(f)g+S_i^+(f)\Delta_i(g)+\delta_i(f)\delta^-_i(g).
\end{aligned}
\ee 
\end{proof}
\subsection{The Effective Wave A}
\begin{lemma}\label{Standard Energy Result}
Let $U$ satisfy \eqref{The efffective equation} with sufficiently smooth initial conditions given by \eqref{Macroscopic Initial Conditions}. Then
\be 
\left|E\left(D^k\frac{\partial^i U}{\partial \tau^i}\right) \right|\leq C\left(\norm{\nabla \phi}_{H^{i+k}}^2+\norm{\psi}_{H^{i+k}}^2\right) \ee where $E\left(D^k\frac{\partial^i U}{\partial \tau^i}\right)$ is defined by \eqref{Continuum Energy}. Here the constant $C$ depends on the wave speed $c$ which in turn is entirely dependent on $\bar{m}.$
\end{lemma}
\begin{remark}
   In the next couple of proofs, we have often opted for the the notation $\frac{d}{d\tau}$ instead of $\frac{\partial}{\partial\tau}$ since we are regarding $E(D^{k}\frac{\partial^{i}U}{\partial\tau^i})$ as function of $\tau$ only. This makes the notation distinct from spatial derivatives. Of course, inside the integral, the notation technically describes a partial derivatives of $U$. 
\end{remark}
\begin{proof}
Note that $D^kU(\X,\tau):\R^2 \times \R \to \R^{2^k}.$ Consider the $j^{\text{th}}$ component of the $2^k$ dimensional vector of $D^kU$ denoted by $D_j^kU $. Consider a ball of radius $r$ in $\R^d$ about the origin denoted $B(r)$. 
\be \label{}  
\dot{E}\left(\frac{\partial^iD_j^k U}{\partial \tau^i}\right)(\tau)=\lim_{r \to \infty}\int_{B(r)} \left(\frac{d^{i+1}D_j^kU}{d\tau^{i+1}} \right) \left(\frac{d^{i+2}D_j^kU}{d\tau^{i+2}}\right) +c^2\nabla\left(\frac{d^{i}D_j^kU}{d\tau^{i}}\right)\cdot \nabla\left( \frac{d^{i+1}D_j^kU}{d\tau^{i+1}}\right)d\X
\ee
Using integration by parts we find 
\be\begin{aligned}\label{Integration by Parts}\dot{E}\left(\frac{\partial^iD_j^k U}{\partial \tau^i}\right)=&\lim_{r \to \infty}\int_{B(r)}\left(\frac{d^{i+1}D_j^kU}{d\tau^{i+1}} \right) \left(\frac{d^{i+2}D_j^kU}{d\tau^{i+2}}\right)-c^2 \Delta\left(\frac{d^iD_j^kU}{d\tau^i}\right)\left( \frac{d^{i+1}D_j^kU}{d\tau^{i+1}}\right) d\X \\&+\lim_{r \to \infty}\int_{\partial B(r)}\left(\frac{\partial\frac{d^iD_j^kU}{d\tau^i}}{\partial \nu}\right)\left(\frac{d^{i+1}D_j^kU}{d\tau^{i+1}}\right)dS.\end{aligned}\ee
Since $U$ solves \eqref{The efffective equation}, swapping the order of partials yields,
\be 
\left(\frac{d^{i+2}D_j^kU}{d\tau^{i+2}}\right)-c^2 \Delta\left(\frac{d^iD^k_jU}{d\tau^i}\right) =0.
\ee

We are left with
\be 
\dot{E}\left(\frac{\partial^iD_j^k U}{\partial \tau^i}\right)(\tau)=\lim_{r \to \infty}\int_{\partial B(r)}\left(\frac{\partial\frac{d^iD_j^kU}{d\tau^i}}{\partial \nu}\right)\left(\frac{d^{i+1}D_j^kU}{d\tau^{i+1}}\right)dS,
\ee
which is $0$ for all $\tau$ due to the finite propagation speed of the wave. Note that at time zero, when $i$ is even, 
\be \label{even} \frac{d^{i}D_j^k U}{d\tau^{i}}=c^2D_j^k\Delta_\X^{\frac{i}{2}}U=c^2D_j^k \Delta_\X^{\frac{i}{2}}\phi\ee
while, when $i$ is odd, 
\be \label{odd} \frac{d^{i} D_j^kU}{d\tau^{i}}=c^2D_j^k\Delta_\X^{\frac{i-1}{2}}\frac{dU
}{d\tau}=c^2D_j^k\Delta_\X^{\frac{i-1}{2}}\psi.\ee
Since \eqref{Integration by Parts} is $0$ and holds for all $j \in \{1, 2, 3,  \cdots, 2^k\}$, we have using either \eqref{even} or \eqref{odd} in \eqref{Continuum Energy} that
\be 
\left|E\left(\frac{\partial^iD_j^k U}{\partial \tau^i}\right)(\tau) \right| = \left|E\left(\frac{\partial^iD_j^k U}{\partial \tau^i}\right)(0) \right| \leq C\left(\norm{\nabla \phi}_{H^{i+k}}^2+\norm{\psi}_{H^{i+k}}^2\right)
\ee 
where the final constant $C$ may depend on $k$ and $c$ (which depends upon $\bar{m}).$
\end{proof}

\begin{lemma}\label{Weighted Energy result}
Let $U$ satisfy \eqref{The efffective equation} with initial conditions given by \eqref{Macroscopic Initial Conditions}. Suppose that $w$ satisfies \eqref{Assumption 1 on Weight} and \eqref{Assumptions 2 on Wieght}. Then 
\be 
\left|E_{w}\left(D^k\frac{\partial^i U}{\partial \tau^i}\right)(\tau)\right| \leq (|\tau|+1)C\left(\norm{\phi}^2_{H_w^{i+k+1}}+\norm{\psi}^2_{H_w^{i+k}}\right).
\ee 
The constant $C$ depends upon $\bar{m}$ as well as the bound $W$ found in \eqref{Assumptions 2 on Wieght}.
\end{lemma}
\begin{proof}
Note that $D^kU(\X,\tau):\R^2 \times \R \to \R^{2^k}.$ Consider the $j^{\text{th}}$ component of the $2^k$ dimensional vector of $D^kU$ denoted by $D_j^kU $. Consider a ball of radius $r$ in $\R^d$ about the origin denoted $B(r)$. 
\be
\dot{E}_{w}\left(\frac{\partial^iD_j^k U}{\partial \tau^i}\right)(\tau)=\lim_{r \to \infty}\int_{B(r)} w^2\left(\frac{d^{i+1}D_j^kU}{d\tau^{i+1}} \right) \left(\frac{d^{i+2}D_j^kU}{d\tau^{i+2}}\right) +c^2w^2\nabla\left(\frac{d^{i}D_j^kU}{d\tau^{i}}\right)\cdot \nabla\left( \frac{d^{i+1}D_j^kU}{d\tau^{i+1}}\right)d\X
\ee
Using integration by parts we find 
\be\begin{aligned}\dot{E}_{w}\left(\frac{\partial^iD_j^k U}{\partial \tau^i}\right)(\tau)=&\lim_{r \to \infty}\int_{B(r)}w^2\left(\frac{d^{i+1}D_j^kU}{d\tau^{i+1}} \right) \left(\frac{d^{i+2}D_j^kU}{d\tau^{i+2}}\right)-c^2 \nabla \cdot\left(w^2\nabla\left(\frac{d^iD_j^kU}{d\tau^i}\right)\right)\left( \frac{d^{i+1}D_j^kU}{d\tau^{i+1}}\right) d\X \\&+\lim_{r \to \infty}\int_{\partial B(r)}w^2c^2\left(\frac{\partial\frac{d^iD_j^kU}{d\tau^i}}{\partial \nu}\right)\left(\frac{d^{i+1}D_j^kU}{d\tau^{i+1}}\right)dS.\end{aligned}\ee
The boundary term vanishes in the limit due to finite propagation speed of the wave. (Decay rates on $U$ and its derivatives are enforced down below.) We need to calculate 
 \be \nabla\cdot\left(w^2\nabla\left(\frac{d^iD_j^kU}{d\tau^i}\right)\right)=\left(\nabla w^2\right)\cdot \nabla \left(\frac{d^iD_j^kU}{d\tau^i}\right)+w^2\Delta\left(\frac{d^iD_j^kU}{d\tau^i}\right).\ee Since $\frac{d^iD_j^kU}{d\tau^i}$ satisfies \eqref{The efffective equation}, we are left with
 \be\dot{E}_w\left(\frac{\partial^iD_j^k U}{\partial \tau^i}\right)(\tau)= \lim_{r \to \infty} \int_{B(r)}\nabla w^2\cdot \nabla \left(\frac{d^{i}D_j^kU}{d\tau^i}\right)\left(\frac{d^{i+1}D_j^kU}{d\tau^{i+1}}\right)d\X.\ee
 
 Using the assumption on $w$ in \eqref{Assumptions 2 on Wieght} and Cauchy-Schwarz, we have  \be \left|\nabla w^2 \cdot \nabla \left(\frac{d^{i}D_j^kU}{d\tau^i}\right)\right| \leq C \left|\nabla\left(\frac{d^{i}D_j^kU}{d\tau^i}\right)\right|.\ee Therefore 
 \be\dot{E}_w\left(\frac{\partial^iD_j^k U}{\partial \tau^i}\right)(\tau)\leq  \lim_{r \to \infty} C\int_{B(r)} \left|\nabla\left(\frac{d^{i}D_j^kU}{d\tau^i}\right)\right|\left|\frac{d^{i+1}D_j^kU}{d\tau^{i+1}}\right|.d\X\ee Using Cauchy-Schwarz once and swapping derivatives
 \be \dot{E}_{w}\left(\frac{\partial^iD_j^k U}{\partial \tau^i}\right)(\tau) \leq C \norm{D_j^{k+1}\left(\frac{d^{i}U}{d\tau^i}\right)}_{L^2}\norm{D_j^k\left(\frac{d^{i+1}U}{d\tau^{i+1}}\right)}_{L^2}. \ee
From \eqref{Basic Energy Bound 1}, we know how to bound such beasts.
\be \dot{E}_{w}\left(\frac{\partial^iD_j^k U}{\partial \tau^i}\right)(\tau) \leq C\left(\norm{\nabla \phi}^2_{H^{i+k}}+\norm{\psi}^2_{H^{i+k}}\right).\ee  Integrating yields
\be 
E_{w}\left(\frac{\partial^iD_j^k U}{\partial \tau^i}\right)(\tau) \leq |\tau| C\left(\norm{\phi}^2_{H^{i+k+1}}+\norm{\psi}^2_{H^{i+k}}\right)+E_{w}\left(D_j^k\frac{\partial^i U}{\partial \tau^i}\right)(0) .
\ee
Using \eqref{even} or \eqref{odd} and the definition \eqref{Def of Weighted Norm w}
\be 
E_{w}\left(\frac{\partial^iD_j^k U}{\partial \tau^i}\right)(0) \leq C\left( \norm{\nabla \phi}_{H^{i+k}_{w}}^2+\norm{\psi}_{H^{i+k}_{w}}^2\right).
\ee
This holds for all $j \in \{1, 2,3, \dots , 2^k\}.$ Thus 
\be 
\left| E_{w}\left(D^k\frac{\partial^i U}{\partial \tau^i}\right)(\tau) \right| \leq (|\tau|+1)C\left(\norm{\phi}^2_{H_w^{i+k}}+\norm{\psi}^2_{H_w^{i+k}}\right).
\ee
Again, the constant $C$ depends upon $k$ and $\bar{m}$ but also $W$.
\end{proof}

\begin{lemma}\label{Tail Energy Lemma}
Let $U$ satisfy \eqref{The efffective equation} with initial conditions given by \eqref{Macroscopic Initial Conditions}
\be
\left|\tilde{E}\left(D^k\frac{\partial^i U}{\partial \tau^i}\right)(\tau)\right| \leq C\left(\norm{\nabla \phi}^2_{H^{i+k}(B(\ep^{-\sigma})^c)}+\norm{\psi}^2_{H^{i+k}(B(\ep^{-\sigma})^c)}\right).
\ee The constant $C$ depends only on $\bar{m}.$
\end{lemma}
\begin{proof}
Without loss of generality, let $\tau$ be positive. Note that $D^kU(\X,\tau):\R^2 \times \R \to \R^{2^k}.$ Consider the $j^{\text{th}}$ component of the $2^k$ dimensional vector of $D^kU$ denoted by $D_j^kU $. We apply Leibniz's rule
\be\begin{aligned}\dot{\tilde{E}}\left(\frac{\partial^iD_j^k U}{\partial \tau^i}\right)(\tau)=&\int_{B(c|\tau|+\ep^{-\sigma} )^c} \left(\frac{d^{i+1}D_j^kU}{d\tau^{i+1}} \right) \left(\frac{d^{i+2}D_j^kU}{d\tau^{i+2}}\right) +c^2\nabla\left(\frac{d^{i}D_j^kU}{d\tau^{i}}\right)\cdot\nabla\left( \frac{d^{i+1}D_j^kU}{d\tau^{i+1}}\right)d\X \\ &-\frac{1}{2}\int_{\partial B(c\tau +\ep^{-\sigma})}c\left(\frac{d^{i+1}D_j^kU}{d\tau^{i+1}} \right)^2+c^3\left|\nabla\left(\frac{d^{i}D_j^kU}{d\tau^{i}}\right)\right|^2dS.
\end{aligned}\ee
Let $\nu$ be the unit normal pointing out of the ball. As we have seen two times already, since $D_j^kU$ satisfies \eqref{The efffective equation}, so integration by parts yields
\be \begin{aligned}
\dot{\tilde{E}}^{(i)}\left(\frac{\partial^iD_j^k U}{\partial \tau^i}\right)(\tau)=&-\int_{\partial B(c\tau +\ep^{-\sigma})}c^2\left(\frac{\partial\frac{d^iD_j^kU}{d\tau^i}}{\partial \nu}\right)\left(\frac{d^{i+1}D_j^kU}{d\tau^{i+1}}\right)
\\&+
\left(\frac{1}{2}c\left(\frac{d^{i+1}D_j^kU}{d\tau^{i+1}} \right)^2+\frac{1}{2}c^3\left|\nabla\left(\frac{d^{i}D_j^kU}{d\tau^{i}}\right)\right|^2\right)dS.
\end{aligned}
\ee
The first term in the integrand is bounded as
\be c^2\left|\left(\frac{\partial\frac{d^iD_j^kU}{d\tau^i}}{\partial \nu}\right)\left(\frac{d^{i+1}D_j^kU}{d\tau^{i+1}}\right) \right| \leq c^2\left|\nabla\frac{d^iD_j^kU}{d\tau^i} \right| \left|\frac{d^{i+1}D_j^kU}{d\tau^{i+1}} \right| \leq \frac{1}{2}c\left(\frac{d^{i+1}D_j^kU}{d\tau^{i+1}} \right)^2+ \frac{1}{2}c^3\left|\nabla\frac{d^iD_j^kU}{d\tau^i} \right|^2.\ee
Therefore 
\be\dot{\tilde{E}}\left(\frac{\partial^iD_j^k U}{\partial \tau^i}\right)(\tau) \leq 0, \ee and thus 
\be \tilde{E}\left(\frac{\partial^iD_j^k U}{\partial \tau^i}\right)(\tau) \leq \tilde{E}\left(D_j^k\frac{\partial^i U}{\partial \tau^i}\right)(0).\ee
Using \eqref{even} or \eqref{odd} 
\be \tilde{E}\left(\frac{\partial^iD_j^k U}{\partial \tau^i}\right)(\tau)\leq  C\left(\norm{\nabla \phi}^{2}_{H^{i+k}(B(\ep^{-\sigma})^c)}+\norm{\psi}^{2}_{H^{i+k}(B(\ep^{-\sigma})^c)}\right) .\ee 
This holds for all $j \in \{1, 2, 3, \dots, 2^k\}$ therefore yielding the result where the final constant will depend upon $k$ and $c$ (which depends on $\bar{m}$.)
\end{proof}

\begin{lemma}\label{Shrinking Tail}
For $\sigma>0$
 \be \norm{\phi}_{H^{k}(B(\ep^{-\sigma})^c)} \leq \ep\norm{\phi}_{H^{k}_{\sigma}} .\ee
\end{lemma}
\begin{proof}
Consider
\be
\norm{D^{j}\phi}^{2}_{L^2(B(\ep^{-\sigma})^c)} =\int_{B(\ep^{-\sigma})^c}\left|D^{j}\phi(\X)\right|^{2}d\X =\int_{B(\ep^{-\sigma})^c}\frac{\left|\ep ^{\sigma}X\right|^{2\sigma^{-1}}}{\left|\ep ^{\sigma}X\right|^{2\sigma^{-1}}}\left|D^{j}\phi(\X)\right|^{2}d\X.\ee 
Since $\ep^{\sigma}|X| \geq 1$ in the region of the integral, we have
\be
\norm{D^{j}\phi}^{2}_{L^2(B(\ep^{-\sigma})^c)} \leq \ep^{2} \int_{B(\ep^{-\sigma})^c}\left|\X\right|^{2 \sigma^{-1}}\left|D^{k}\phi(\X)\right|^{2}d\X. 
\ee
Thus 
\be
\norm{D^{j}\phi}^2_{L^2(B(\ep^{-\sigma})^c)} \leq \ep^{2}\norm{\left(1+\left| \ \cdot \ \right| \right)^{\sigma^{-1}}D^{j}\phi(\cdot)}^2_{L^2}.
\ee
Taking square roots, summing over $j$ and comparing with the definition of $H^{k}_\sigma$ in \eqref{Def of Weighted Norm} yields the result.
\end{proof}
\subsection{Residual Estimates A}
\begin{lemma}\label{Riemann Sum}
Let $f(\x):\R^2 \to \R $ be in $H^2$ and let $\x^{(\bj)} \in \R^2\cap \prod_{i=1}^{2}[j_1,j_1+1]$ where $\bj \in \mathcal{Q} \subset \Z^2$. Then 
\be 
\sum_{\bj \in \mathcal{Q}}f(\x^{(\bj)})^2 \leq \sum_{\bj \in \mathcal{Q}} 4\norm{f}^2_{H^2(\prod_{i=1}^2 [\bj,\bj+\e_i])}.
\ee
\end{lemma}
\begin{proof}
 Note that $f$ is continues by Sobolev embedding. Let $\underline{\x}^{(\bj)}$ denote the minimizer of $f$ over $[j_1,j_1+1]\times\{x^{(\bj)}_2\}$. We can write 
 \be 
 f(\x^{(\bj)})^2=f(\underline{\x}^{(\bj)})^2+\int_{\underline{x}^{(\bj)}_1}^{x^{(\bj)}_1}\frac{d}{dx_1}f(x_1,x^{(\bj)}_2)^2dx_1.
 \ee 
 We can apply Cauchy-Schwarz to get
 \be 
  f(\x^{(\bj)})^2 \leq f(\underline{\x}^{(\bj)})^2 +2 \left|\int_{\underline{x}^{(\bj)}_1}^{x^{(\bj)}_1}f(x_1,x^{(\bj)}_2)^2dx_1\right|^{\frac{1}{2}}\left|\int_{\underline{x}^{(\bj)}_1}^{x^{(\bj)}_1}\left(\frac{d}{dx_1}f(x_1,x^{(\bj)}_2)\right)^2dx_1 \right|^{\frac{1}{2}},
 \ee 
 which becomes 
 \be 
 f(\x^{(\bj)})^2 \leq f(\underline{\x}^{(\bj)})^2 + \left|\int_{\underline{x}^{(\bj)}_1}^{x^{(\bj)}_1}f(x_1,x^{(\bj)}_2)^2dx_1\right|+\left|\int_{\underline{x}^{(\bj)}_1}^{x^{(\bj)}_1}\left(\frac{d}{dx_1}f(x_1,x^{(\bj)}_2)\right)^2dx_1\right| .
 \ee
 Since the integrands are non-negative
 \be 
 f(\x^{(\bj)})^2 \leq f(\underline{\x}^{(\bj)})^2 + \int_{j_1}^{j_1+1}f(x_1,x^{(\bj)}_2)^2dx_1+\int_{j_1}^{j_1+1}\left(\frac{d}{dx_1}f(x_1,x^{(\bj)}_2)\right)^2dx_1.
 \ee
 The first term on the right hand side is smaller than the first integral. 
 \be \label{Final Inequality for U}
 f(\x^{(\bj)})^2 \leq  2\int_{j_1}^{j_1+1}f(x_1,x^{(\bj)}_2)^2dx_1+\int_{j_1}^{j_1+1}\left(\frac{d}{dx_1}f(x_1,x^{(\bj)}_2)\right)^2dx_1.
 \ee
 Let 
 \be \label{Def of F}
 F(x_2)^2\coloneqq 2\int_{j_1}^{j_1+1}f(x_1,x_2)^2dx_1+\int_{j_1}^{j_1+1}\left(\frac{d}{dx_1}f(x_1,x_2)\right)^2dx_1.
 \ee
 Since we are done with the old $\underline{x}^{(\bj)}_2$, we use it to denote the minimizer of $F$ over $[j_2,j_2+1].$ Now similar to before we get that 
 \be \label{Initial equality for F}
 F(x_2^{(\bj)})^2 = F(\underline{x}^{(\bj)}_2)^2 + \int_{\underline{x}^{(\bj)}_2}^{x_2^{(\bj)}}\frac{d}{dx_2}F(x_2)^2dx_2.
 \ee
 Now
 \be 
 \frac{d}{dx_2}F(x_2)^2=\int_{j_1}^{j_1+1}4f(x_1,x_2)\frac{\partial}{\partial x_2}f(x_1,x_2)+2\frac{\partial}{\partial x_1}f(x_1,x_2)\frac{\partial^2}{\partial x_2x_1}f(x_1,x_2)dx_1.
 \ee
 One application of Cauchy-Schwarz to the integrand gives
 \be 
  \left|\frac{d}{dx_2}F(x_2)^2\right| \leq 4\int_{j_1}^{j_1+1}\left(f^2+\left(\frac{\partial }{\partial x_1}f\right)^2\right)^{\frac{1}{2}}\left(\left(\frac{\partial }{\partial x_2}f\right)^2+\left(\frac{\partial^2 }{\partial x_2x_1}f\right)^2\right)^{\frac{1}{2}}dx_1,
 \ee
which becomes
 \be 
  \left|\frac{d}{dx_2}F(x_2)^2\right| \leq 2\int_{j_1}^{j_1+1}f^2+\left(\frac{\partial }{\partial x_1}f\right)^2+\left(\frac{\partial }{\partial x_2}f\right)^2+\left(\frac{\partial^2 }{\partial x_2x_1}f\right)^2dx_1.
 \ee
 Therefore 
 \be 
 \label{Final Inequality for F der}
 \int_{\underline{x}^{(\bj)}_2}^{x_2^{(\bj)}}\left|\frac{d}{dx_2}F^2\right|dx_2 \leq  2\int_{j_2}^{j_2+1}\int_{j_1}^{j_1+1}f^2+\left(\frac{\partial }{\partial x_1}f\right)^2+\left(\frac{\partial }{\partial x_2}f\right)^2+\left(\frac{\partial^2 }{\partial x_2x_1}f\right)^2dx_1dx_2.
 \ee
 Also
 \be \label{Final inequality for F}
 F(\underline{x}_2^{(\bj)})^2 \leq \int_{j_2}^{j_2+1}F(x_2)^2dx_2.
 \ee
 By the \eqref{Def of F}
 \be \label{Final inequality for F min} 
 F(\underline{x}_2^{(\bj)})^2  \leq 2\int_{j_2}^{j_2+1} \int_{j_1}^{j_1+1}f(x_1,x_2)^2+\left(\frac{\partial}{\partial x_1}f(x_1,x_2)\right)^2dx_1dx_2.
 \ee
 Using \eqref{Final Inequality for F der} and \eqref{Final inequality for F min} in \eqref{Initial equality for F}, we have 
 \be 
 F(x_2^{(\bj)})^2 \leq \int_{j_2}^{j_2+1}\int_{j_1}^{j_1+1}4f^2+4\left(\frac{\partial }{\partial x_1}f\right)^2+2\left(\frac{\partial }{\partial x_2}f\right)^2+2\left(\frac{\partial^2 }{\partial x_2x_1}f\right)^2dx_1dx_2.
\ee
From \eqref{Final Inequality for U}, we have 
\be 
f(\x^{(\bj)})^2 \leq  \int_{j_2}^{j_2+1}\int_{j_1}^{j_1+1}4f^2+4\left(\frac{\partial }{\partial x_1}f\right)^2+2\left(\frac{\partial }{\partial x_2}f\right)^2+2\left(\frac{\partial^2 }{\partial x_2x_1}f\right)^2dx_1dx_2.
\ee
Summing over $\bj \in \mathcal{Q}$ we get
\be 
\sum_{\bj \in \mathcal{Q}}f(\x^{(\bj)})^2 \leq 4\sum_{\bj \in \mathcal{Q}} \norm{f}^2_{H^2(\prod_{i=1}^2 [\bj,\bj+\e_i])}.
\ee
\end{proof}
\begin{corollary}\label{Obvious Corollary}
 In the same context as Lemma \ref{Riemann Sum} with $\mathcal{Q}=\Z^2,$ we have 
 \be 
 \sum_{\bj \in \Z^2} f(\ep\x^{(\bj)})^2 \leq  \ep^{-2}4\norm{f}^2_{H^2}.
 \ee 
\end{corollary}
\begin{lemma}\label{Tail Sum}
 In the same context as Lemma \ref{Riemann Sum} with $\mathcal{Q}=D(\0,\ep^{-1}R+1)^c,$ we have 
 \be 
 \sum_{\bj \in \mathcal{Q}} f(\ep\x^{(\bj)})^2 \leq  \ep^{-2}4\norm{f}^2_{H^2(B(R)^c)}.
 \ee 
\end{lemma}
\begin{proof}
From Lemma \ref{Riemann Sum}
\be 
\sum_{\bj \in D(\0,\ep^{-1}R+1)^c}f(\ep\x^{(\bj)})^2 \leq 4\sum_{\bj \in D(\0,\ep^{-1}R+1)^c } \norm{f(\ep \cdot)}^2_{H^2(\prod_{i=1}^2[\bj,\bj+\e_i])}.
\ee
Note that \be \bigcup_{\bj \in D(\0,\ep^{-1}R+1)^c}\prod_{i=1}^2[\bj,\bj+\e_i]\subset B(\ep^{-1}R)^c .\ee Hence 
\be 
\sum_{\bj \in D(\0,\ep^{-1}R+1)^c } \norm{f(\ep \cdot)}^2_{H^2(\prod_{i=1}^2[\bj,\bj+\e_i])} \leq \norm{f(\ep\cdot)}^2_{H^2( B(\ep^{-1}R)^c)}.
\ee
Written in polar coordinates 
\be
\begin{aligned}
\norm{f(\ep \cdot)}^2_{H^2( B(\ep^{-1}R)^c)} =&\int_0^{2\pi}\int_{\ep^{-1}R}^{\infty}f(\theta, \ep r)^2rdrd\theta\\=&\ep^{-2}\int_0^{2\pi}\int_{R}^\infty f(\theta,s)^2sdsd\theta\\=&\ep^{-2}\norm{f}^2_{H^2(B(R)^c)}.
\end{aligned}
\ee
\end{proof}
\begin{lemma}\label{FT}
Let $\delta^\pm_{\ep i}\phi(\X) \coloneqq \pm(\phi( \X)-\phi( \X-\ep\e_i)).$ Let $w$ be as in assumptions $\eqref{Assumption 1 on Weight}$ and \eqref{Assumptions 2 on Wieght}. Then 
\be \norm{\delta^\pm_{\ep i}(\phi)}_{H^k_w} \leq \ep C\norm{\phi}_{H^{k+1}_w}.\ee
\end{lemma}
\begin{proof}
We do the proof for $i=1$. Let $\partial$ be some generic partial derivative of order $k$. By the Fundamental Theorem of Calculus
\be 
\partial \delta^+_{\ep i}\phi(\X)= \delta^+_{\ep i}\partial\phi(\X)=\int_{ X_1}^{X_1+\ep}\partial_{X'_1}\partial\phi(X_1',X_2)dX_1'.
\ee
 Now we calculate 
\be \begin{aligned}
&\int_{-\infty}^{-\infty}\int_{-\infty}^{-\infty}w(X_1,X_2)^2\left(\partial\delta_{\ep i}\phi(X_1,X_2)^2\right)^2dX_1dX_2 \\ &=\int_{-\infty}^{-\infty}\int_{-\infty}^{-\infty}w(X_1,X_2)^2\left(\int_{ X_1}^{ X_1+\ep}\partial_{X'_1}\partial\phi(X_1',X_2)dX_1'\right)^2dX_1dX_2
\end{aligned}
\ee
An application of Jensen's Inequality yields 
\be \begin{aligned}
&\int_{-\infty}^{-\infty}\int_{-\infty}^{-\infty}\left(\int_{X_1}^{ X_1+\ep}\partial_{X'_1}\partial\phi(X_1',X_2)dX_1'\right)^2dX_1dX_2 \\&= \ep \int_{-\infty}^{-\infty}\int_{-\infty}^{-\infty}w(X_1,X_2)^2\int_{ X_1}^{ X_1+\ep}\left(\partial_{X'_1}\partial\phi(X_1',X_2)\right)^2dX_1'dX_1dX_2.
\end{aligned}
\ee
Now we flip the $X_1$ and $X_1'$
\be
\begin{aligned}
&\int_{-\infty}^{-\infty}\int_{-\infty}^{-\infty}w(X_1,X_2)^2\int_{ X_1}^{ X_1+\ep}\left(\partial_{X'_1}\partial\phi(X_1',X_2)\right)^2dX_1'dX_1dX_2\\&=\int_{-\infty}^{-\infty}\int_{-\infty}^{-\infty}\left(\partial_{X'_1}\partial\phi(X_1',X_2)\right)^2\int_{X_1'-\ep}^{X_1'}w(X_1,X_2)^2dX_1dX_1'dX_2
\end{aligned}
\ee
From our definition of $w$ in \eqref{Assumption 1 on Weight}
\be 
\begin{aligned}
&\int_{-\infty}^{-\infty}\int_{-\infty}^{-\infty}\left(\partial_{X'_1}\partial\phi(X_1',X_2)\right)^2\int_{X_1'-\ep}^{X_1'}w(X_1,X_2)^2dX_1dX_1'dX_2\\ &\leq \ep\int_{-\infty}^{-\infty}\int_{-\infty}^{-\infty}\left(\partial_{X'_1}\partial\phi(X_1',X_2)\right)^2w(X_1',X_2)^2dX_1'dX_2.
\end{aligned}
\ee
\end{proof}
Stringing these inequalities together, we find 
\be 
\begin{aligned}
&\int_{-\infty}^{-\infty}\int_{-\infty}^{-\infty}w(X_1,X_2)^2\left(\partial\delta_{\ep i}\phi(X_1,X_2)^2\right)^2dX_1dX_2 \\ &\leq \ep^2\int_{-\infty}^{-\infty}\int_{-\infty}^{-\infty}\left(\partial_{X'_1}\partial\phi(X_1',X_2)\right)^2w(X_1',X_2)^2dX_1'dX_2,
\end{aligned}
\ee
which implies 
\be 
\norm{\delta^+_{\ep i}\phi}_{H^k_w} \leq \ep \norm{\phi}_{H^{k+1}_w}.
\ee
Now let $S^-_{\ep i}\phi(\X) \coloneqq \phi(\X-\ep \e_i)$. Thus $\delta_{\ep i}^-\phi=S_{\ep i}^-\delta^+_{\ep i}\phi$. We now show that $S_{\ep i}^{-}$ is a bounded operator. Consider any $f \in L^2$, then
\be \int_{\R^2}w(\X)^2f(\X-\ep \e_i)^2d\X=\int_{\R^2}w(\X+\ep \e_i)^2f(\X)^2d\X. \ee

From the definition of the weight in \eqref{Assumption 1 on Weight}, there exists a constant $C$ s.t. 
\be 
w(\X+\ep \e_i)^2 \leq C^2w(\X)^2.
\ee
Therefore the operator $S_{\ep i}^-$ is bounded by this constant $C.$ Thus 
\be 
\norm{\delta^-_{\ep i}\phi}_{H^k_w} \leq \ep C \norm{\phi}_{H^{k+1}_w}.
\ee
\begin{corollary}\label{FTC}
 Let $\Delta_{\ep i} \phi(\X)\coloneqq \phi(\X+\ep \e_i)-2\phi(\X)+\phi(\X-\ep \e_i)$ and $w$ as in the previous lemma. Then 
 \be
 \norm{\Delta_{ \ep} \phi i}_{H_w^k}\leq \ep^2 C\norm{\phi}_{H^{k+2}_w}.
 \ee
\end{corollary}

\subsection{Coarse Graining A}
\begin{lemma}\label{Interpolation}
Let $U: \R^2 \to \R^2 $ be in $H^s$ with $s>1$. Put $U_\ep(\x) \coloneqq U(\ep \x).$ Then 
\be 
\lim_{\ep \to 0^+} \norm{\mathcal{LS}[U_\ep](\cdot/\ep)-U}_{L^2}=0
\ee
where $\mathcal{L}$ and $\mathcal{S}$ are defined by \eqref{low pass} and \eqref{Sampling}. 
\end{lemma}
\begin{proof}
From their definitions we have 
\be 
\mathcal{LS}[U_{\ep}](\x)=\frac{1}{4\pi}\int_{[-\pi,\pi]^2}\left(\sum_{\bj \in \Z}e^{-i \y\cdot \bj }U(\ep \bj)\right) e^{i\y \cdot \x}d\y.
\ee
Changing variables from $\y \to \ep \y $ gives 
\be 
\mathcal{LS}[U_{\ep}](\x)=\frac{\ep^2}{4\pi}\int_{[-\pi/\ep,\pi/\ep]^2}\left(\sum_{\bj \in \Z}e^{-i \ep \y\cdot \bj }U(\ep \bj)\right) e^{i \ep \y \cdot \x}d\y.
\ee
Now swap the integral and the sum and integrate to get
\be 
\mathcal{LS}[U_{\ep}](\x)=\sum_{\bj \in \Z^2}U(\ep \bj)\sinc(x_1- j_1)\sinc(x_2-j_2),
\ee
 where $\sinc$ is the normalized $\sinc$ function, $\frac{\sin(\pi x)}{\pi x}.$ Subbing $\x \to \X/\ep$ yields 
\be 
\mathcal{LS}[U_{\ep}](\X/\ep)=\sum_{\bj \in \Z^2}U(\ep \bj)\sinc(X_1/\ep- j_1)\sinc(X_2/\ep- j_2).
\ee
This series is exactly equal to the band-limited approximation of $U_\ep$ given by
\be \mathcal{LS}[U_{\ep}](\X/\ep)=\tilde{U}_\ep(\X) \coloneqq \mathcal{F}^{-1}[\theta_{\pi/\ep}\mathcal{F}[U]](\X). \ee
See \cite{Stenger} for details regarding band-limited functions. Now we only need to compare $U$ and $\tilde{U}_\ep$, but by use of Plancherel's Theorem, this is equivalent to looking at 
\be \begin{aligned} 
\norm{\theta_{\pi/\ep} \mathcal{F}[U]-\mathcal{F}[U]}^2_{L^2}&=\int_{|\y|_{\infty} > \pi/\ep } \hat{U}(\y)^2d\y \\
&\leq \sup_{|\y|_\infty \geq \pi/\ep}\left( 
\frac{1}{|\y|^{2s}}\right) \int_{|\y|_\infty
 >\pi /\ep 
}|\y|^{2s}\hat{U}(\y)^2 d\y \\ 
&\leq C\ep^{2s}\norm{U}^2_{H^s}.
\end{aligned}
\ee
Since we have taken $s> 1$, the result is proved and the rate of convergence is greater than $\ep.$

\end{proof}


\begin{thebibliography}{10}
\bibitem{Brillouin} 
L. Brillouin, \emph{Wave propagation in periodic structures}, Dover Books on Engineering and Engineering Physics, Dover Publications, Inc., New York, New York, 1953. 
\bibitem{Cioranescu}
 D. Cioranescu and P. Donato, \emph{An Introduction to Homogenization}, Oxford Lecture Ser.
Math. Appl. 17, The Clarendon Press, Oxford University Press, New York, New York, 1999.

\bibitem{Chirilus-Bruckner}
M. Chirilus-Bruckner, C. Chong, O. Prill, and G. Schneider, \emph{Rigorous description of
macroscopic wave packets in infinite periodic chains of coupled oscillators by modulation
equations}, Discrete Contin. Dyn. Syst. Ser. S, 5 (2012), pp. 879–901.

\bibitem{Craig}
W. Craig, \emph{A Course on Partial Differential Equations}, Graduate Studies in Math. 197 American Mathematical Society, Providence, Rhode Island, 2018, pp. 122-123.

\bibitem{Duffin}
R. J. Duffin, \emph{Basic Properties of Discrete Analytic Functions}, Duke Math. J., 23 (1956), pp. 335 - 363.

\bibitem{Durrett}
R. Durrett, \emph{Probability, Theory and Examples}, Cambride Ser. in Stat. and Prob. Math., Cambridge University Press, New York, 2010, pp. 65.

\bibitem{Evans}
L. C. Evans, \emph{Partial Differential Equations, Second Edition}, Graduate Studies in Math. Vol 19, American Mathematical Society, Providence, Rhode Island, 2010.

\bibitem{Papanicolau}
J.-P. Fouque, J. Garnier, G. Papanicolaou, K. S{\o}lna, \emph{Wave Propagation and Time Reversal in Randomly Layered Media}, Stochastic Modeling and Applied Probability 56, Springer Science+Business Media, New York, New York, 2007.

\bibitem{Gaison}
J. Gaison, S. Moskow, J. D. Wright, and Q. Zhang, \emph{Approximation of Polyatomic FPU Lattices by KdV
Equations}, Mult. Scale Model. Simul., 12 (2014), pp. 953-995.


\bibitem{Spohn}
J. Lukkarinen and H. Spohn, \emph{ Kinetic limit for wave propagation in a random media}, Arch. Rational Mech. Anal. 183 (2007) 93–162.


\bibitem{Porter}
M. J. Mart\'inez, P. G. Kevrekidis, and M. A. Porter. \emph{Superdiffusive tansport and energy localization in disordered granular crystals}, Phys Rev. E 93 022902 (2016).

\bibitem{Massart}
P. Massart, \emph{Concentration Inequalities and Model Selction}, Ecole d'Et\'e de Probabilit\'es de Saint-Flour XXXIII-2003, Springer Verlag, Berlin, 2007, pp. 21-23.

\bibitem{McGinnis}
J. A. McGinnis and J. D. Wright, \emph{Using Random Walks to Establish Wavelike Behavior in an FPUT
System with Random Coefficients}, Discrete Contin. Dyn. Syst. Ser. S, 5 (2021).

\bibitem{Stenger}
J. McNamee, F. Stenger and E. L. Whitney, \emph{Whittaker's Cardinal Function in Retrospect}, Mathematics of Computation, 25 (1971). pp 141-154.

\bibitem{Mielke}
A. Mielke, \emph{Macroscopic Behavior of Microscopic Oscillations in Harmonic Lattices via Wigner-Husimi Transforms}, Arch. Rational Mech.Anal., 181 (2006), pp. 401–448.

\bibitem{Schneider}
G. Schneider and C. E. Wayne, \emph{Counter-propagating waves on fluid surfaces and the continuum limit of the Fermi-Pasta-Ulam model}, International Conference on Differential
Equations, Vols. 1, 2 (Berlin, 1999), World Scientific, River Edge, NJ, 2000, pp. 390–404.







\end{thebibliography}
\end{document}